\documentclass{article}
\usepackage[utf8]{inputenc}

\usepackage{amssymb}
\usepackage{amsthm}
\usepackage{amsmath,amscd}
\usepackage[mathscr]{euscript}
\usepackage[all]{xy}
\usepackage{lmodern}
\usepackage[T1]{fontenc}
\usepackage{stmaryrd} 
\usepackage[textwidth=14cm,hcentering]{geometry}
\usepackage[colorlinks=true,linkcolor=red,citecolor=blue]{hyperref}
\usepackage{cleveref}
\usepackage{mathtools}
\usepackage{xspace}
\usepackage{tikz}
\usetikzlibrary{cd}
\usetikzlibrary{arrows,positioning}
\usepackage{bbm}
\usepackage{aliascnt}
\usepackage{caption}

\title{Chromatic Noshift}
\author{Maxime Ramzi }
\date{}

\newtheorem{thm}{Theorem}[section]
\newaliascnt{lm}{thm}  
\newtheorem{lm}[lm]{Lemma}
\aliascntresetthe{lm}
\Crefname{lm}{Lemma}{Lemmas}
\newaliascnt{prop}{thm}  
\newtheorem{prop}[prop]{Proposition}
\aliascntresetthe{prop}
\Crefname{prop}{Proposition}{Propositions}
\newaliascnt{cor}{thm}  
\newtheorem{cor}[cor]{Corollary}
\aliascntresetthe{cor}
\Crefname{cor}{Corollary}{Corollaries}

\newtheorem*{thm*}{Theorem}
\newtheorem*{cor*}{Corollary}

\theoremstyle{definition}
\newaliascnt{defn}{thm}  
\newtheorem{defn}[defn]{Definition}
\aliascntresetthe{defn}
\Crefname{defn}{Definition}{Definitions}
\newaliascnt{cons}{thm}  
\newtheorem{cons}[cons]{Construction}
\aliascntresetthe{cons}
\Crefname{cons}{Construction}{Constructions}
\newaliascnt{nota}{thm}  
\newtheorem{nota}[nota]{Notation}
\aliascntresetthe{nota}
\Crefname{nota}{Notation}{Notations}
\newaliascnt{conv}{thm}  
\newtheorem{conv}[conv]{Convention}
\aliascntresetthe{conv}
\Crefname{conv}{Convention}{Conventions}
\newaliascnt{ex}{thm}  
\newtheorem{ex}[ex]{Example}
\aliascntresetthe{ex}
\Crefname{ex}{Example}{Examples}
\newaliascnt{rmk}{thm}  
\newtheorem{rmk}[rmk]{Remark}
\aliascntresetthe{rmk}
\Crefname{rmk}{Remark}{Remarks}
\newaliascnt{ques}{thm}  
\newtheorem{ques}[ques]{Question}
\aliascntresetthe{ques}
\Crefname{ques}{Question}{Questions}
\newaliascnt{conj}{thm}  
\newtheorem{conj}[conj]{Conjecture}
\aliascntresetthe{conj}
\Crefname{conj}{Conjecture}{Conjectures}
\newaliascnt{warn}{thm}  
\newtheorem{warn}[warn]{Warning}
\aliascntresetthe{warn}
\Crefname{warn}{Warning}{Warnings}
\newaliascnt{obs}{thm}  
\newtheorem{obs}[obs]{Observation}
\aliascntresetthe{obs}
\Crefname{obs}{Observation}{Observations}
\newtheorem*{ques*}{Question}
\newtheorem*{rmk*}{Remark}
\newtheorem*{ex*}{Example}
\newaliascnt{recoll}{thm}  
\newtheorem{recoll}[recoll]{Recollection}
\aliascntresetthe{recoll}
\Crefname{recoll}{Recollection}{Recollections}

\AtBeginEnvironment{defn}{\pushQED{\qed}}
\AtEndEnvironment{defn}{\popQED\enddefn}
\AtBeginEnvironment{defn*}{\pushQED{\qed}}
\AtEndEnvironment{defn*}{\popQED\enddefn}
\AtBeginEnvironment{cons}{\pushQED{\qed}}
\AtEndEnvironment{cons}{\popQED\endcons}
\AtBeginEnvironment{assu}{\pushQED{\qed}}
\AtEndEnvironment{assu}{\popQED\endassu}
\AtBeginEnvironment{nota}{\pushQED{\qed}}
\AtEndEnvironment{nota}{\popQED\endnota}
\AtBeginEnvironment{nota*}{\pushQED{\qed}}
\AtEndEnvironment{nota*}{\popQED\endnota}
\AtBeginEnvironment{conv}{\pushQED{\qed}}
\AtEndEnvironment{conv}{\popQED\enddefn}\AtBeginEnvironment{ex}{\pushQED{\qed}}
\AtEndEnvironment{ex}{\popQED\endex}
\AtBeginEnvironment{exs}{\pushQED{\qed}}
\AtEndEnvironment{exs}{\popQED\endexs}
\AtBeginEnvironment{exc}{\pushQED{\qed}}
\AtEndEnvironment{exc}{\popQED\endexc}
\AtBeginEnvironment{rmk}{\pushQED{\qed}}
\AtEndEnvironment{rmk}{\popQED\endrmk}
\AtBeginEnvironment{rmk*}{\pushQED{\qed}}
\AtEndEnvironment{rmk*}{\popQED\endrmk}
\AtBeginEnvironment{ques}{\pushQED{\qed}}
\AtEndEnvironment{ques}{\popQED\endques}
\AtBeginEnvironment{ques*}{\pushQED{\qed}}
\AtEndEnvironment{ques*}{\popQED\endques}
\AtBeginEnvironment{conj}{\pushQED{\qed}}
\AtEndEnvironment{conj}{\popQED\endconj}
\AtBeginEnvironment{warn}{\pushQED{\qed}}
\AtEndEnvironment{warn}{\popQED\endwarn}
\AtBeginEnvironment{obs}{\pushQED{\qed}}
\AtEndEnvironment{obs}{\popQED\endobs}

\newcommand{\op}{^{\mathrm{op}}}
\newcommand{\ho}{\mathrm{ho}}
\newcommand{\cat}{\mathrm}
\newcommand{\Cat}{\cat{Cat}}
\newcommand{\Catperf}{\mathrm{Cat}^{\mathrm{perf}}}
\newcommand{\on}{\operatorname}
\newcommand{\id}{\mathrm{id}}
\newcommand{\Fun}{\on{Fun}}
\newcommand{\map}{\on{map}}
\newcommand{\Map}{\on{Map}}

\newcommand{\NN}{\mathbb N}
\newcommand{\Z}{\mathbb Z}
\newcommand{\Sph}{\mathbb S}
\newcommand{\can}{\mathrm{can}}

\newcommand{\tr}{\mathrm{tr}}

\newcommand{\An}{\mathrm{An}}
\newcommand{\Sp}{\cat{Sp}}
\newcommand{\at}{\mathrm{at}}

\newcommand{\PrL}{\mathrm{Pr}^\mathrm{L} }

\newcommand{\Alg}{\mathrm{Alg}}
\newcommand{\CAlg}{\mathrm{CAlg}}

\newcommand{\Mod}{\cat{Mod}}

\newcommand{\HH}{\mathrm{HH}}
\newcommand{\THH}{\mathrm{THH}}
\newcommand{\Perf}{\mathrm{Perf}}

\newcommand{\Ind}{\mathrm{Ind}}
\newcommand{\End}{\mathrm{End}}

\newcommand{\pt}{\mathrm{pt}}
\newcommand{\colim}{\mathrm{colim}}

\newcommand{\Pic}{\mathrm{Pic}}

\newcommand{\PPic}{\mathbf{Pic}}

\newcommand{\Mot}{\mathrm{Mot}}

\newcommand{\one}{\mathbbm{1}}
\newcommand{\V}{\mathcal V}

\newcommand{\st}{\mathrm{st}}
\newcommand{\dbl}{\mathrm{dbl}}

\newcommand{\udl}[1]{\lfloor #1 \rfloor}

\newcommand{\Dim}{\mathrm{Dim}}
\newcommand{\Vect}{\mathrm{Vect}}
\newcommand{\Cob}{\mathrm{Cob}}
\newcommand{\FCom}{\mathrm{Free}_{\mathbb{E}_{\infty}}}
\newcommand{\Q}{\mathbb{Q}}
\newcommand{\Frrig}{\mathrm{Free}^\mathrm{rig}}
\newcommand{\rig}{\mathrm{rig}}
\newcommand{\TwoRing}{2\mathrm{Ring}}
\newcommand{\locrig}{\mathrm{locrig}}
\newcommand{\category}{\xspace{$\infty$-category}\xspace}
\newcommand{\categories}{\xspace{$\infty$-categories}\xspace}

\begin{document}

\maketitle\begin{abstract}
The chromatic redshift philosophy, introduced by Ausoni and Rognes, suggests that algebraic $K$-theory raises chromatic height by $1$. We show that the analogue of this philosophy fails in the case of rigid symmetric monoidal stable $\infty$-categories. More precisely, we construct examples of rigid $T(n)$-local categories $C$ where a refinement $\Dim$ of the dimension morphism induces an equivalence $$K(C)\to \End(\one_C)^{BS^1}$$ and for which $K(C)$ therefore vanishes $T(n+1)$-locally. 

In fact, we prove that this equivalence always holds for $\aleph_1$-Nullstellensatzian rigid $T(n)$-local categories in the sense of Burklund, Schlank and Yuan. We study more in depth the rational version of these results to find a rigid rational additive $1$-category witnessing the failure of redshift at height $0$. 

Finally, we use our methods to prove and generalize a conjecture of Levy about categorification of ordinary rings.
\end{abstract}
\begin{center}
    \includegraphics[scale=0.4]{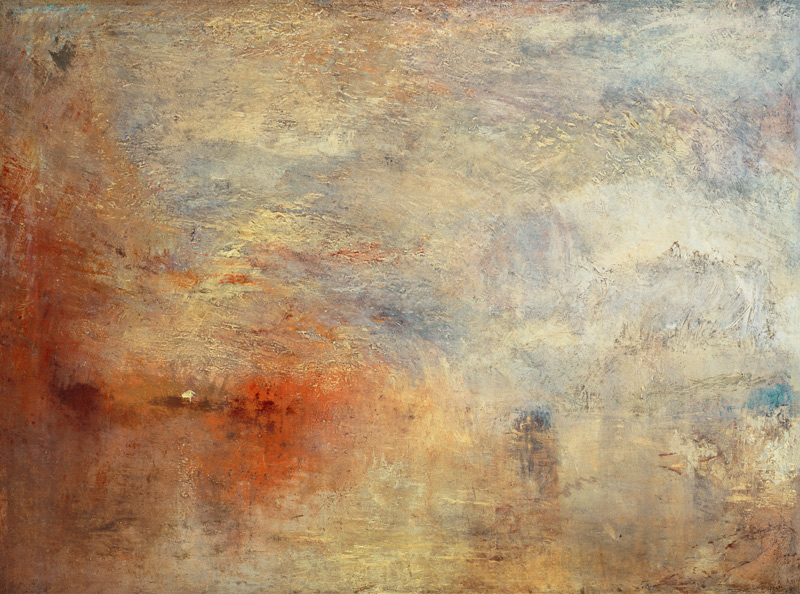}
    \captionof{figure}{Sun Setting over a Lake, J.M.W. Turner}
\end{center}
\tableofcontents
\setcounter{secnumdepth}{0}

\section{Introduction}
The study of the algebraic $K$-theory of ring spectra and its relation to the chromatic filtration, initiated by Waldhausen \cite{waldhausen}, has been significantly transformed in recent years by Ausoni and Rognes' redshift philosophy \cite{AusoniRognes,Rognes}. This philosophy roughly suggests that algebraic $K$-theory raises chromatic height by $1$, and has become an organizing principle in the study of chromatically localized $K$-theory. 

It has since been substantially confirmed, for example in \cite{KK1AR,hahnwilson,Kroots}, and a precise form of it was recently completely proved in a series of landmark papers \cite{CMNN,LMMT,Allenredshift,ChroNS} which, altogether, establish chromatic redshift for commutative ring spectra: the $K$-theory of a commutative ring spectrum of height $n\geq 0$ has height exactly $n+1$. 

A very surprising feature of these recent advances is that while $K$-theory is fundamentally a categorical invariant, the results in \cite{Allenredshift,ChroNS}, which establish the lower bound for redshift (the $K$-theory of a commutative ring spectrum of height $n\geq 0$ is of height \emph{at least} $n+1$) work exclusively at the ring-theoretic level. This is in stark contrast with the usual $K$-theoretic practice of working at the categorical level, as in \cite{CMNN,LMMT}, where the upper bound for redshift is established globally, for stable $\infty$-categories (the $K$-theory of a stable $\infty$-category of height $\leq n$ is of height $\leq n+1$). This leaves open the question of the lower bound for redshift for symmetric monoidal stable $\infty$-categories. It is easy to see that one cannot expect a fully general redshift: 
\begin{ex*}
    The \category of spectra with countable homotopy groups is a small symmetric monoidal \category and its $K$-theory vanishes. This example can be modified to work at any finite height. 
\end{ex*}
However, it is well understood that arbitrary symmetric monoidal stable $\infty$-categories are somewhat ill-behaved and too general; the ones that are thought of as behaving somewhat like categories of perfect modules over commutative ring spectra are the \emph{rigid} symmetric monoidal stable $\infty$-categories, that is, those where every object is dualizable. These categories are particularly interesting and somewhat unavoidable -- in \cite{CMNN} the proof of Galois descent for the $T(n+1)$-local $K$-theory of $T(n)$-local Galois extensions uses a category of $T(n)$-locally dualizable objects over a commutative ring spectrum $A$, and in particular establishes that \emph{it} behaves $K$-theoretically like the category of perfect modules over $A$.

Among other things, for rigid symmetric monoidal stable $\infty$-categories, there is a \emph{dimension morphism} $K(C)\to \End(\one_C)$ which guarantees that $K(C)$ is of height \emph{at least} greater or equal to that of $\End(\one_C)$, where $\one_C$ is the unit of $C$, which is a weak form of the lower bound for redshift (``$K$-theory does not \emph{decrease} height''). This puts rigid categories at an interesting place for the redshift question: they are too constrained for the general swindle-type constructions to give obvious counterexamples, and there is \emph{some} evidence that they behave somewhat like categories of perfect modules. Therefore one might expect the $K$-theory of rigid stably symmetric monoidal \categories to also exhibit chromatic redshift. 

Our main results shows that this expectation fails completely:
\begin{thm*}[Noshift, \Cref{cor:noshift}]
    Let $C$ be a nonzero rigid $T(n)$-local symmetric monoidal \category. There exists a nonzero rigid $T(n)$-local symmetric monoidal \category and a symmetric monoidal functor (which can be chosen to be fully faithful) $C\to D$ such that $$L_{T(n+1)}K(D^\dbl) = 0.$$
\end{thm*}
This result shows that the distinction between the categorical results of \cite{CMNN,LMMT} and the ring-theoretic ones from \cite{Allenredshift,ChroNS} is essential: the redshift phenomenon for commutative ring spectra does not arise from a purely categorical property of algebraic $K$-theory.
\begin{rmk*}
Since $D$ is something ``like'' $T(n)$-local spectra, there are several small categories one can associate to it, and so ``the $K$-theory of $D$'' can a priori be slightly ambiguous. In fact, $T(n+1)$-locally, all reasonable notions agree (\Cref{lm:agreeK}) and so the above result does not suffer from this ambiguity. 
\end{rmk*}
\begin{rmk*}
    The $D$'s that we construct are a priori quite large, and so one may wonder how ``prevalent'' this failure of redshift is, e.g. if it ever occurs for relatively ``small'' rigid $T(n)$-local symmetric monoidal \categories. 

    It is actually not so difficult to deduce from the above result that there exist $\aleph_1$-compact, i.e. \emph{countably generated}, rigid $T(n)$-local symmetric monoidal \categories with the same $T(n+1)$-local vanishing. We explain this in \Cref{cor:smallnoshift}. Since this Noshift phenomenon is ``upwards closed'' (if it is satisfied by $D$, and $E$ receives a map from $D$, then it is also satisfied by $E$), this suggests that failure of redshift is closer to the norm, and that commutative ring spectra are in fact very specific among rigid symmetric monoidal \categories. 
\end{rmk*}

To prove this result, we study how one can modify the $K$-theory of rigid categories. We find that a certain refinement $\Dim : K(C)\to \End(\one_C)^{BS^1}$ of the dimension morphism governs precisely which classes in $K$-theory can be added or killed under rigid extensions $C\to D$, that is, fully faithful symmetric monoidal exact functors. We prove precise versions of this idea, but the outcome is the following:
\begin{thm*}[\Cref{thm:designer}]
    Let $C$ be a locally rigid $2$-ring. There exists a rigid extension $C\to D$ such that the Dimension morphism $$\Dim: K(D^\dbl)\to \End(\one_D)^{BS^1} \simeq \End(\one_C)^{BS^1}$$ is an equivalence. 
\end{thm*}
The term ``locally rigid $2$-ring'' will be defined in \Cref{defn:locrig}, and is here to allow examples such as the category of $T(n)$-local spectra, which is not rigidly-compactly generated. In that example, considering rigid algebras over $\Sp_{T(n)}$ allows us to consider rigid categories with $T(n)$-local mapping spectra. We note that while there is a more general definition, and the results proved in this paper would apply in that larger generality as well, here, all locally rigid $2$-rings are taken to be compactly generated. 

In turn, the Dimension morphism  $\Dim: K(C)\to \End(\one_C)^{BS^1}$ appearing in the theorem is a refinement of the usual dimension morphism, recording the $S^1$-action on symmetric monoidal dimensions\footnote{The convention here is that ``Dimension'' records the $S^1$-action, while ``dimension'' does not.}. This extra structure on dimensions arises from the structure of the free rigid symmetric monoidal category on an object, per the $1$ dimensional cobordism hypothesis \cite{baezdolan} (see \Cref{section:dimensions} for a detailed discussion), and it is ultimately why it appears in this result. 
\begin{rmk*}
    One can think of this result as a ``maximal no-shift'' result: while we usually think of $K$-theory as not being ``linear'' over the base, this theorem says that this expectation can be maximally falsified. 
\end{rmk*}
As a consequence, we prove and generalize a conjecture of Levy \cite[Conjecture 5]{levycat} about the categorification of ordinary commutative rings, thereby also answering Question 16 in \textit{loc. cit.}:
\begin{cor*}[\Cref{cor:levytext}]
 For every ordinary ring $R$, there exists a small rigid idempotent-complete commutative $\Perf(R)$-algebra $C$ such that $K_0(C)\cong R$.    
\end{cor*}

We were led to these results by trying to follow the proof strategy of \cite{ChroNS}: the authors introduce the notion of a Nullstellensatzian object in a category, which abstracts Hilbert's Nullstellensatz for algebraically closed fields, and they prove two key facts about them. The first, which is relatively formal and works in great generality, is that any object has a map to some nonzero Nullstellensatzian object. The second, which is much less formal, is that in the category of $T(n)$-local commutative ring spectra, Nullstellensatzian objects are precisely the Lubin--Tate theories (based on algebraically closed fields). 

More generally, this means that studying Nullstellensatzian objects provides a reasonable strategy to address these kinds of questions: if they satisfy a form of chromatic redshift (as proved by Yuan \cite{Allenredshift} in the above example), then it follows that the same holds for all objects; and if not, they provide canonical counterexamples. Our results imply that they do not satisfy redshift, thus providing canonical examples of this failure:
\begin{thm*}[$K$-theoretic Nullstellensatz, \Cref{thm:KNS,lm:end1NS,cor:NSnoshift}]
 Let $C$ be a locally rigid $2$-ring, e.g. $C=\Sp_{T(n)}$, with $\kappa$-compact unit, and let $D$ be a $\max(\kappa, \aleph_1)$-Nullstellensatzian rigid commutative $C$-algebra. The Dimension morphism $$\Dim: K(D^\dbl)\to \End(\one_D)^{BS^1}$$ is an \emph{equivalence} of commutative ring spectra, and $\End(\one_D)$ is a $\max(\kappa,\aleph_1)$-Nullstellensatzian commutative algebra in $C$.

 In particular, if $C=\Sp_{T(n)}$, then $L_{T(n+1)}K(D^\dbl)=0$.
\end{thm*}
\begin{ex*}
 When $C=\Sp_{T(n)}, n \geq 0$ (and some implicit prime $p$ if $n\geq 1$), the second condition and \cite{ChroNS} imply that $\End(\one_D)$ is a Lubin--Tate theory of height $n$, based on an algebraically closed field of characteristic $p$ (or $0$, if $n=0$). Therefore, the target of the Dimension morphism is completely understandable and this is in fact a complete calculation of $K(D^\dbl)$. 
\end{ex*}
    When $C=\Mod_{\mathbb F_p}$ for some odd prime $p$, the situation is more subtle, see \cite{FloNS} for a discussion.
    \begin{rmk*}
    For a fully general such $C$, we are not able to guarantee the existence of such Nullstellensatzian $D$'s, because in general the $0$ algebra in $C$ need not be compact.  However, as we shall see in \Cref{cor:existNS}, this existence \emph{is} guaranteed as soon as $C$ is what we call ``quasi-compact'', cf. \Cref{defn:quasicompact}, which covers most cases of interest (e.g. it covers all rigid $C$'s, and localizations of rigid $C$'s at dualizable objects, such as $\Sp_{T(n)}$). Thus the above theorem is not vacuous in the $T(n)$-local setting.
\end{rmk*}

So far, we have only discussed $K$-theory, but given \cite{ChroNS} it only seems reasonable to try and probe Nullstellensatzian rigid \categories in more detail and ask whether we can classify or understand them to some extent. In general, this seems like a challenging task. At height $0$, that is, in the rational setting, one key additional tool makes this study more approachable: free rigid categories were already studied by Deligne, who proved that his $\mathrm{Rep}(GL_t)$ category, a variant of the free rational rigid category on an object, is semi-simple. This rather strong result allows for a more detailed study.  

We therefore conclude the paper with the beginnings of such an investigation over $\mathbb Q$. We find that Nullstellensatzian rigid rational \categories are quite large (which is to be expected), and have ``large objects'' making them really different from ``fields'' in any sense. To get a sense of what they look like, we mention the following result: 
\begin{thm*}[\Cref{lm:uniquetrans}, \Cref{cor:summandtrans}, \Cref{cor:NStransdiv}]
Let $C$ be an $\aleph_1$-Nullstellensatzian rigid rational stably symmetric monoidal \category. 

Objects in $C^\dbl$ whose dimension is \emph{not} an integer are classified by their Dimension in $\End(\one_C)^{BS^1}$ up to \emph{equivalence}, their $K$-theory classes span $K_0(C^\dbl)$, and for each such $x$, and an arbitrary $z\in C^\dbl$, $z$ is a summand of $x$. 

Furthermore, each object with non-integer dimension is $\oplus$-divisible: for every $n\geq 1$, it is of the form $\oplus_n y$ for some $y$. 
\end{thm*}

Using this, we are further able to refine the failure of redshift at height $0$ and show that it also happens for ordinary $1$-categories - thus, this failure is not only a higher algebraic phenomenon: 
\begin{cor*}[{\Cref{cor:1cat}}]
     There exist nonzero rational additively symmetric monoidal $1$-categories where every object is dualizable whose (additive) algebraic $K$-theory is rational. 
 \end{cor*}
\begin{rmk*}
As is the case for the Noshift theorem, these categories can be chosen to be countably generated. Over $\mathbb Q$, this simply means countable. 
\end{rmk*}
\begin{rmk*}
The question of redshift is interesting beyond the commutative case: it was fully established for commutative ring spectra as explained in the beginning of the introduction, and our results establish its failure in the case of rigid categories; but the question remains interesting for $\mathbb E_k$-ring spectra, $k\geq 2$ and even for some associative ring spectra, even though in this case we know that we have to restrict the class of rings that are considered. There is a growing body of work around these questions (see e.g. \cite{KK1AR,hahnwilson,Kroots}), and the present work adds nothing to this question. 
\end{rmk*}

\subsection*{Proof strategy}
We use two key inputs to prove our main result: first, the main result of \cite{RSW} which allows us to categorify $K$-theory classes, and second, the description of free rigid \categories from \cite{Jan,Frrig}, a generalization of the $1$D cobordism hypothesis, which allows us to pass from $\infty$-categories with no monoidal structure to rigid $\infty$-categories and control the result in terms of traces and Dimensions.

Concretely, given a rigid symmetric monoidal stable $\infty$-category $C$ and a homotopy class in $\End(\one_C)^{BS^1}$, we interpret that class as a map $\Sigma^n \one_C[BS^1]\to \one_C$ and interpret the source in terms of the free rigid commutative $C$-algebra on some $C$-module categorifying suspensions or deloopings (\Cref{cor:unitinFrrig}). This allows us to lift that class along the Dimension map to the $K$-theory of a small rigid extension of $C$ (\Cref{lm:atomicsurj}). 

Conversely, given a homotopy class in $K(C^\dbl)$, \cite{RSW} lets us interpret its image in $K(D^\dbl)$ as a map of $D$-modules (\Cref{cor:almostrealizingclasses}) for a small rigid extension $D$ of $C$, and we relate properties of the induced map on free rigid commutative algebras with the application of $\Dim$ to the original homotopy class (\Cref{cor:commutativediagram}). In particular, if $\Dim$ vanishes on the original class, this interpretation allows us to kill the class \emph{motivically} in some small enough extension of $D$ (\Cref{lm:atinj}). 

A fair amount of the paper is set-up, dedicated to making precise the above strategy. A minor difficulty is related to cardinality estimates, and this is why we reason in terms of homotopy classes instead of reasoning directly at the level of spectra. A slightly more relevant difficulty is that in the $T(n)$-local setting ($n>0$), which we generalize to the locally rigid setting, $K$-theory is harder to relate to localizing motives, and the results of \cite{RSW} apply to the latter more than to the former. Thus we have to be a bit careful in actually implementing the above strategy. 

As briefly mentioned above, in the rational case, we add another method to analyze Nullstellensatzian rational rigid $2$-rings, namely, we leverage the semi-simplicity of Deligne's $\mathrm{Rep}(GL_t)$ category \cite{deligne}. Semi-simplicity guarantees the non-vanishing of many pushouts of rigid $\infty$-categories (\Cref{cor:semiscons}), which is generally a difficult task. Since knowing when something is or is not zero is crucial to understanding Nullstellensatzian objects, this guarantee is extremely helpful. We use this to give a second, simpler proof of our Noshift theorem at height $0$, and to observe that it can be implemented for rigid additive $1$-categories.     
\subsection*{Relation to other work}
We drew a lot of inspiration from \cite{ChroNS}, and the idea to consider the redshift question for Nullstellensatzian rigid \categories arose from there. This is ultimately what led to our ``$K$-theoretic Nullstellensatz'', and later to the more refined results presented in the introduction.

The author originally envisioned a program where one could prove a ``rigid version'' of the main result from \cite{RSW} and try to prove redshift this way. Clearly, this failed, but the failure was robust enough to actually disprove redshift. One can see the results of this paper as approximations to such a ``rigid version'' of \cite{RSW}. 

Once the main results in the Nullstellensatzian case had been proved, the connection to \cite{levycat} also became clear since \Cref{thm:KNS} immediately implied a positive answer to Conjecture 5 in \emph{loc. cit.} (without idempotent-completeness), and trying to answer Question 16 in \emph{loc. cit.} \emph{with} idempotent-completeness led to the more precise results in \Cref{section:main}. Our methods are, however, quite different from the ones in \cite{levycat}. 

Finally, there is a philosophical connection between this work and the work of Barthel, Hyslop and the author \cite{BHR}, as in both cases the main results were obtained by analyzing the free situation in detail. We believe that, especially in characteristic $0$ but also more generally, it is a fruitful endeavour to try and push these ideas to understand (rational) tt-geometry more precisely: free objects and Nullstellensatzian objects do not necessarily arise daily in the common practice of tt-geometry, but they are an important part of the landscape, and in somewhat dual ways they control the theory. In this landscape, the work of Barkan and Steinebrunner \cite{Jan} can serve as a guiding light to understand these free constructions (see also \cite{Frrig}).
\subsection*{Sectionwise outline}
\Cref{section:prelim} contains various conventions, recollections and preliminaries that set the stage for the rest of the paper:
\begin{itemize}
    \item \Cref{section:catalg} reminds the reader of the notions of rigid and locally rigid $2$-rings, and basic manipulations about enriched hom objects and tensor products of categories;
    \item \Cref{section:Mot} is about motives: it sets up our conventions in that regard, recalls and mildly refines some results from \cite{RSW}, and finally warns about the pitfalls associated with motives over locally rigid $2$-rings;
    \item \Cref{section:dimensions} is about the $1$D cobordism hypothesis, basic facts about the $S^1$-action it induces on dimensions of dualizable objects in $2$-rings, and Barkan--Steinebrunner's extension of the $1$D cobordism hypothesis \cite{Jan}.
\end{itemize} 
\Cref{section:main} is the main section of the paper, it contains our technical results and the actual implementation of the strategy described in the introduction. There, we prove \Cref{thm:designer} and its consequence, the Noshift theorem (\Cref{cor:noshift}).

\Cref{section:NS} recalls the notion of Nullstellensatzian object from \cite{ChroNS}, establishes basic properties of Nullstellensatzian rigid categories, and finally proves our $K$-theoretic Nullstellensatz (\Cref{thm:KNS}). 

Finally, \Cref{section:Q} contains the slightly more in-depth analysis of Nullstellensatzian rational rigid categories: it starts with a specialization of the $S^1$-actions on dimensions in the rational\footnote{More honestly, complex oriented.} case; \Cref{section:sscat} contains some brief recollections about semi-simple additive categories; \Cref{section:nonint} contains the main results in height $0$ and introduces a particular class of objects we call ``non-integral'', which are the center of our approach to Nullstellensatzian rigid rational $2$-rings; \Cref{section:duality} studies the action of duality on non-integral objects; and finally \Cref{section:picard} studies invertible objects in rational $2$-rings. This section also contains an elementary proof of the failure of redshift at height $0$, so the reader is invited to consult it for the easiest example\footnote{Which, historically, was the first example we found.} of our main result.
\subsection*{Conventions}
We use the language of $\infty$-categories as developed in \cite{HTT,HA}. From now on, we say ``categories'' for $\infty$-categories, and specify with ``ordinary'' or ``$1$-'' when we are referring to ordinary categories. All relevant notions are to be understood in the $\infty$-categorical sense. 

Besides the conventions from \textit{loc. cit.} which we try to follow, we have the following conventions: 
\begin{enumerate}
\item Instead of ``space'' or ``$\infty$-groupoid'', we say ``anima'' and write $\An$ for the corresponding category;
\item We use $\Map$ for mapping anima, and $\map$ for mapping spectra in stable categories;
    \item For a symmetric monoidal category $C$, we let $\CAlg(C)$ denote commutative algebras in $C$, and $\PrL$ or $\PrL_\st$ are always understood to be equipped with the Lurie symmetric monoidal structure; 
    \item Given a symmetric monoidal category $C$, we let $\one_C$ denote the unit, and for a commutative algebra $A$ in $C$, $\Mod_A(C)$ the category of $A$-modules. We let $C^\dbl$ denote the full subcategory spanned by dualizable objects in $C$; 
    \item For $C\in\CAlg(\PrL_\st)$, we let $\one_C\otimes-: \Sp\to C$ denote the unique symmetric monoidal colimit-preserving functor, and we write $\udl{-}: C\to \Sp$ for its right adjoint, $\Map_C(\one_C,-)$; 
    \item Given a symmetric monoidal category $C$ and a $C$-module $M$ which is $C$-enriched, we write $\hom$ for the $C$-enriched hom in $M$, hoping that the $C$ under consideration is clear from context, and similarly for $\End$ (which, when $C$ is stable, can therefore refer to either an endomorphism $C$-object or an endomorphism ring spectrum). 
    \end{enumerate}
 
\subsection*{Acknowledgments}
I started thinking about this subject again after dealing with Deligne's $\mathrm{Rep}(GL_t)$ category in the context of my collaboration with Tobias Barthel and Logan Hyslop, and I thank them both for the inspiration this brought. I also thank Logan for some helpful feedback.

I am also grateful to Shay Ben-Moshe, Robert Burklund, Shai Keidar, Akhil Mathew, Thomas Nikolaus, Vova Sosnilo for helpful conversations related to this project, and particularly to Jan Steinebrunner for his patience and for generously describing the results of \cite{Jan}. The proof in the rational case was inspired by a trick I learned in the Almost Mathematics Seminar at Münster Universität, ran by Franziska Jahnke and Thomas Nikolaus, and I thank them for organizing this seminar.  Finally, I received helpful feedback about a draft from Phil Pützstück.

This research was funded by the Deutsche Forschungsgemeinschaft (DFG, German Research Foundation) – Project-ID 427320536 – SFB 1442, as well as under Germany's Excellence Strategy EXC 2044/2 –390685587, Mathematics Münster: Dynamics–Geometry–Structure.

\setcounter{secnumdepth}{3}
\section{Preliminaries}\label{section:prelim}
\subsection{Categorical algebra}\label{section:catalg}
In this section, we introduce our basic objects of study, namely (compactly generated) locally rigid $2$-rings, and prove some of their basic properties and properties of their modules that we will need. The story simplifies psychologically when we consider rigid $2$-rings instead, since in that case the connection between the ``large'' world of cocomplete categories and the ``small'' world is tighter, and in the small world, rigid $2$-rings are quite familiar objects.  In fact, even for rigid $2$-rings, some of the proofs we give simplify when passing through the large world, so while we will make the explicit connection between large and small, the reader is invited to think of all our categories as large, unless we explicitly consider categories of dualizable objects or compact objects. 

First, we recall some definitions, and we refer to \cite{LocRig} for a detailed treatment. 
\begin{nota}
    We let $\TwoRing$ denote $\CAlg(\PrL_\st)$, and call its objects $2$-rings. 
\end{nota}
\begin{defn}\label{defn:locrig}
A compactly generated $2$-ring is called \emph{locally rigid} if all its compact objects are dualizable. It is called \emph{rigid} if furthermore its unit is compact, or equivalently if all its dualizable objects are compact. 

We let $\TwoRing^\rig$ denote the full subcategory of $\CAlg(\PrL_\st)$ spanned by rigid $2$-rings, and $\TwoRing^{\locrig}$ the non-full subcategory of $\CAlg(\PrL_\st)$ spanned by locally rigid $2$-rings and functors preserving compact objects. 
\end{defn}
\begin{rmk}
    The above is the compactly generated version of the notion of rigid from \cite[Chapter 1, §9]{GR}, also studied in \cite[Section 2.2]{HSSS} and \cite{LocRig}. 

    All the arguments in the paper actually apply to the usual, ``dualizable'' version of the notion, using the generalization to dualizable categories of the results of Barkan and Steinebrunner \cite{Jan} from \cite{Frrig}. Since we actually want to say something about the ``small''/compactly-generated world\footnote{To make sure that the examples we construct are not just obtained by broadening the definition of rigid categories.}, we have chosen to write the paper in this context; and to avoid clutter, we do not state everything twice. 
\end{rmk}
\begin{rmk}
    That the above definitions are equivalent (in the compactly generated case) to \cite[Definition 4.5, Definition 4.34]{LocRig} follows from \cite[Example 4.6]{LocRig} and \cite[Example 1.24]{Dbl}. 
\end{rmk}
\begin{rmk}
    If $C$ is a rigid $2$-ring, its category of compact objects $C^{\aleph_0}$ is equivalently its category of dualizable objects and this is a small stable category with all objects dualizable, hence a(n enhancement of) a tt-category à la Balmer.  

    If $C$ is only locally rigid, $C^{\aleph_0}$ is a non-unitally symmetric monoidal stable category, with an ind-unit. On the other hand, $C^\dbl$ is a small $2$-ring as above, but its ind-completion is generally a ``de-completed'' version of $C$ (e.g. if $C=\Mod_E(\Sp_{T(n)})$, where $E$ is Morava $E$-theory, then $\Ind(C^\dbl)\simeq \Mod_E(\Sp)$). 
\end{rmk}

\begin{ex}
    If $R$ is a commutative ring spectrum, the category $\Mod_R$ of $R$-module spectra is a rigid $2$-ring. 
\end{ex}
\begin{ex}
    The category of $p$-complete spectra, $\Sp_p$, is a locally rigid $2$-ring which is not rigid. 
\end{ex}

For our purposes, we will need a relative version. We will only care about things that are rigid relative to a locally rigid category, for which we can give the following definition:
\begin{defn}
    Let $C\to D$ be a morphism of locally rigid $2$-rings. It is called rigid, or $D$ is called a rigid commutative $C$-algebra, if its right adjoint $D\to C$ preserves colimits. 
\end{defn}
\begin{rmk}
    We defined morphisms in $\TwoRing^\locrig$ as preserving compact objects, and hence, since we defined the objects as being compactly generated, this condition is automatic. In other words, ``rigid morphism of locally rigid $2$-rings'' is synonymous with ``morphism in $\TwoRing^\locrig$'', and we use it to disambiguate this phrase from ``morphism in $\CAlg(\PrL_\st)$ between locally rigid $2$-rings''. 
\end{rmk}
\begin{rmk}
    This is equivalent to the relative notion of rigid commutative algebra from \cite[Definition 4.5]{LocRig}  by Proposition 4.18 in \textit{loc.cit.}. 
\end{rmk}
\begin{nota}
Let $C$ be a locally rigid $2$-ring. 
    We let $\TwoRing^\rig_C$ denote $\TwoRing^\locrig_{C/}$, or equivalently, the full subcategory of $\CAlg(\PrL_\st)_{C/}$ spanned by rigid commutative $C$-algebras. 
\end{nota}
\begin{nota}
    If $C=\Mod_\Q$, we simply write $\TwoRing^\rig_\Q$, and if $C=\Sp_{T(n)}$, we write $\TwoRing^\rig_{T(n)}$ for simplicity. 
\end{nota}
\begin{rmk}\label{rmk:relrig}
   If $C$ is a rigid $2$-ring, $\TwoRing^\rig_C \simeq (\TwoRing^\rig)_{C/}$. More generally, if $C$ is a locally rigid $2$-ring and $C\to D$ is rigid, then $\TwoRing_D^\rig\simeq (\TwoRing^\rig_C)_{D/}$. 
\end{rmk}
\begin{rmk}
   Any map between rigid commutative $C$-algebras preserves compacts and $C$-atomics, to be defined below.
\end{rmk}
\begin{nota}
For $C$ a locally rigid $2$-ring, we let $\Mod_{C,\omega}$ denote the category of compactly generated $C$-modules and compact-preserving functors between them. 
\end{nota}
\begin{rmk}
    Via taking compact objects, this category is equivalent to the full subcategory of non-unital modules over the non-unital commutative algebra $C^{\aleph_0}$ in $\Catperf$ such that they become unital upon ind-completing. 

    Under this equivalence, rigid commutative $C$-algebras correspond to non-unital commutative algebras $D$ over $C^{\aleph_0}$ in $\Catperf$,  whose ind-completion is unital and such that the map $D\to \Ind(D)$ lands in dualizable objects. 
\end{rmk}
\begin{nota}
    We let $\Frrig_C : \Mod_{C,\omega}\to \TwoRing^\rig_C$ denote the left adjoint to the forgetful functor. Its existence is easy in the case of rigid $2$-rings, and can be deduced from that in general -- alternatively, it exists by the main result of \cite{Frrig}. 
\end{nota}

The following notion will also be relevant: 
\begin{defn}
    Let $C$ be a $2$-ring, and $M$ a $C$-module (e.g. a commutative $C$-algebra). An object $m\in M$ is called \emph{atomic} if the right adjoint $\hom(m,-): M\to C$ to the $C$-linear functor $C\xrightarrow{-\otimes m}M$ preserves colimits, and is $C$-linear, that is, the canonical map $c\otimes \hom(m,-)\to \hom(m,c\otimes -)$ is an equivalence for all $c\in C$. 
\end{defn}
\begin{rmk}
  If $C$ is locally rigid, the second condition is automatic, see \cite[Proposition 4.18]{LocRig}. 
\end{rmk}
We make this notion explicit in the following examples: 
\begin{ex}\label{ex:rigat=cpct}
    If $C$ is rigid, the notion of atomic object coincides with that of compact object, cf. \cite[Corollary 3.8]{LocRig}. 
\end{ex}
\begin{ex}\label{ex:rigat=dbl}
    If $C\to D$ is a rigid map of locally rigid $2$-rings, $C$-atomic objects in $D$ coincide with dualizable objects in $D$: for $C=D$ this is essentially definitional, and in general, see \cite[Lemma 4.50]{LocRig}. 
\end{ex}
\begin{ex}
    If $C$ is a locally rigid $2$-ring, compact objects in (dualizable) $C$-modules are all atomic, cf. \cite[Proposition 4.15]{LocRig}. 
\end{ex}
\begin{lm}
    Suppose $C$ is a rigid $2$-ring, and fix $E\in C^\dbl$. If the localization of $C$ at $E$, $C_E$, is compactly generated, then it is locally rigid. In that case, for a $C_E$-module $M$, an object $m\in M$ is atomic if and only if $E^\vee \otimes m$ is compact. 
\end{lm}
\begin{proof}
We note that the functor $E\otimes - : C\to C$ inverts $E$-equivalences, and therefore factors uniquely through a \emph{colimit-preserving} functor $E\otimes - : C_E\to C$, given by forgetting to $C$ and then tensoring with $E$. 

We also note that $\hom(m,-)\otimes E \simeq \hom(E^\vee\otimes m,-)$ as functors to $C_E$ or to $C$. 

Therefore, if $m$ is atomic, then $\hom(m,-)\otimes E\simeq \hom(E^\vee\otimes m-,)$ preserves colimits with values in $C$, and hence $E^\vee\otimes m$ is $C$-atomic, i.e. compact by \Cref{ex:rigat=cpct}. 

Conversely, if $E^\vee\otimes m$ is compact, then $\hom(E^\vee\otimes m,-) \simeq \hom(m,-)\otimes E$ commutes with colimits with values in $C$, and hence with values in $C_E$, but in the latter, tensoring with $E$ is both colimit-preserving and conservative, so $\hom(m,-)$ also preserves colimits. 
\end{proof}
Thus we may unwind the notion of atomic object in our motivating example of a locally rigid ``base'', namely, $T(n)$-local spectra: 
\begin{ex}\label{ex:loccompactT(n)}
Fix a prime $p$ and an integer $n\geq 1$. 
    Let $C$ be the category of $L_n^f$-local spectra, $L_n^f\Sp$ - this is rigid since it is a localization of spectra at dualizable objects and hence is equivalent to $\Mod_{L_n^f\Sph}(\Sp)$. In this case, for any finite type $n$ spectrum $V$, $L_n^fV$ is dualizable and the $L_n^fV$-localization of $L_n^f\Sp$ is $\Sp_{T(n)}$, the category of $T(n)$-local spectra. In particular, if $M$ is a $\Sp_{T(n)}$-module, $m\in M$ is atomic if and only if for any finite type $n$-spectrum, $V\otimes m$ is compact. 
\end{ex}

Having introduced our basic objects of study, let us establish some basic properties which will be used throughout.

\begin{lm}\label{lm:homsintensor}
Let $C$ be a $2$-ring, $M,N$ are two $C$-modules, $m_0,m_1\in M, n_0,n_1\in~N$. If $m_0,n_0$ are atomic, then, with internal homs calculated in $C$, the natural map $$\hom_M(m_0,m_1)\otimes \hom_N(n_0,n_1)\to \hom_{M\otimes_C N}(m_0\otimes n_0,m_1\otimes n_1)$$ is an equivalence. 
\end{lm}
\begin{proof}
Since $m_0\in M$ and $n_0\in N$ are atomic, the $C$-linear functors $C\xrightarrow{m_0\otimes -}M$, and $C\xrightarrow{n_0\otimes-}N$ classifying $m_0,n_0$ respectively have $C$-linear colimit-preserving right adjoints by assumption.  Hence, the right adjoint of their tensor product over $C$ is the tensor product of their right adjoints since $\otimes_{C}$ is a $2$-functor, \cite[§4.4]{HSS}.

Since internal homs are exactly right adjoints to these classifying functors, unwinding this statement gives the desired result. 
\end{proof}
The following is a special case:
\begin{cor}\label{cor:unittensor}
Let $C\to C_0,C_1$ be two maps of $2$-rings, and assume $C$ is rigid, and that the units in $C_0,C_1$ are compact. More generally, assume simply that $C$ is a $2$-ring and that the units in $C_0,C_1$ are $C$-atomic. 

In this case, in $C$ we have\footnote{Both the endomorphism objects and the tensor product are taken in $C$!} $\End(\one_{C_0})\otimes\End(\one_{C_1})\simeq \End(\one_{C_0\otimes_C C_1})$.
\end{cor}
The following is a direct consequence of \Cref{lm:homsintensor}: 
\begin{cor}\label{cor:fftensor}
    Let $C$ be a $2$-ring, and $M,N,P$ be $C$-modules. Suppose $M,N,P$ are atomically generated in the sense that their atomic objects generate them under colimits and tensors with $C$. 

    In this case for any atomic-preserving fully faithful $C$-linear functor $M\to N$, the induced functor $M\otimes_C P\to N\otimes_C P$ is also fully faithful.  
\end{cor}
A special case of the above is if $C$ is a locally rigid $2$-ring and $M,N$ are compactly generated $C$-modules, then it suffices to check fully faithfulness on compact objects. 

\begin{defn}\label{defn:ext}
    We say a rigid map $C\to D$ of locally rigid $2$-rings is an \emph{extension}, or that $D$ is a \emph{rigid extension} of $C$ if it is fully faithful. 
\end{defn}
We introduce two pieces of notation that will be convenient throughout: 
\begin{nota}\label{nota:basechange}
Let $C$ be a $2$-ring, and $R\in\CAlg(C)$. Let $M,N$ be two $\Mod_R(C)$-modules. If $C$ is understood from context, we abbreviate $M\otimes_{\Mod_R(C)}N$ as $M\otimes_R N$. If furthermore $N=\Mod_S(C)$ for some commutative $R$-algebra $S$, we also write $M\otimes_R S$ for $M\otimes_R\Mod_S(C)$. 
\end{nota}
We will not need the following a lot, but to avoid confusion, let us give the following name to the ``small'' version of rigid $2$-rings:
\begin{nota}\label{nota:smallrigid}
    We call small stably symmetric monoidal categories with all objects dualizable ``small rigid $2$-rings''. The dictionary between small rigid $2$-rings and rigid $2$-rings is simply that ind-completing small rigid $2$-rings gives rigid $2$-rings, and full subcategories of compact/dualizable objects in rigid $2$-rings are small rigid $2$-rings.
\end{nota}
Since $\Ind(-)$ establishes an equivalence between (idempotent complete) small rigid $2$-rings and rigid $2$-rings\footnote{Which, we recall, are compactly generated in the context of this article.}, and between (small, idempotent complete) modules over the former and compactly generated modules over the latter, every notion we introduce for rigid $2$-rings specializes to a notion for small rigid $2$-rings, and we will do the translation implicitly at times. No confusion should arise, since except $0$, no stable category is both cocomplete and small. 

\begin{rmk}
    Most of the results of this paper have generalizations to non-locally rigid $2$-rings, by considering the general notion of rigid commutative algebra over a general $2$-ring, cf. \cite{LocRig}. We have opted to nonetheless stay in the context of locally rigid $2$-rings, mostly for readability (it allows us to discuss compact objects rather than atomic objects), but also because the main examples where we truly care about the distinction between ``small rigid $C^\dbl$-algebras'' and ``rigid $C$-algebras'' are ``completions'' of rigid categories, which are locally rigid. 
\end{rmk}
\subsection{Motives}\label{section:Mot}
In this section, we briefly discuss motives over locally rigid categories and related subtleties, recall the main result from \cite{RSW} as well as some mild variants of the intermediate results therein which we shall need. 

We refer to \cite[Appendix B]{RSW} for the definition of $C$-motives when $C$ is an arbitrary $2$-ring, as well as to \cite{efimovrigidity} for a more in-depth discussion of the locally rigid case.  Since we will use slightly different notation, let us introduce it as follows, in the restricted context of locally rigid $2$-rings: 
\begin{recoll}
    For $C$ a locally rigid $2$-ring, we have a stable presentably symmetric monoidal category $\Mot_C$ of $C$-linear motives. It comes with a functor $M_C:\Mod_{C,\omega}\to \Mot_C$ which preserves filtered colimits and sends fiber/cofiber sequences to fiber/cofiber sequences. 

    $\Mot_{-}$ is a functor on the category of $2$-rings, and $M_{-}$ is natural along the extension of scalars functoriality.

    If $C$ is rigid, $\map_{\Mot_C}(M_C(C), M_C(-))$ is equivalent to $K$-theory of the compact objects (or continuous $K$-theory of the module itself, following \cite{efimovI}). 
\end{recoll}
The last fact here fails when $C$ is only locally rigid, and ultimately it comes from the following point: when $C$ is rigid, the extension of scalars/restriction of scalars adjunction along $\Sp\to C$ descends to the category of motives sinces both the unit $M\to C\otimes M$ and the counit $C\otimes M\to M$ preserve compacts, but if $C$ is only locally rigid, only the counit does. 

Another way of understanding the issue is that the functor $$\Mod_{C,\omega}\to \Catperf, D\mapsto D^\at$$ does not preserve fiber/cofiber sequences. 

Therefore, ``maps from the unit'' in $\Mot_C$ is more subtle than $K$-theory when $C$ is only locally rigid. Efimov beautifully explains in \cite[Section 5.4]{efimovrigidity} how to describe it in categorical terms, at least when the unit of $C$ is $\aleph_1$-compact, but for our purposes, this description is not needed. Despite this difficulty, in \Cref{section:main} we will manage to say something about $K$-theory of dualizable objects in $C$ in some cases (as well as endomorphisms of the unit in $\Mot_C$). 
\begin{rmk}
    If we name the functor $\map_{\Mot_C}(M_C(C),M_C(-))$ ``relative $K$-theory'', then let us point out that the above argument shows also that if $C\to D$ is a rigid map of locally rigid $2$-rings\footnote{This is not necessary, but we have not introduced in this article the general notion of a rigid map of $2$-rings. See \cite{LocRig} for a discussion.}, then $K$-theory relative to $D$ agrees with that relative to $C$. 
\end{rmk}
This suggests the notation:
\begin{nota}
    For $C$ a locally rigid\footnote{Cf. the above footnote.} $2$-ring, let $K_C$ denote $\map_{\Mot_C}(M_C(C),M_C(-)):\Mod_{C,\omega}\to \Sp$.
\end{nota}
\begin{cons}\label{cons:compK}
For $C$ a locally rigid\footnote{Cf. the above footnote.} $2$-ring, there is a natural transformation on $\Mod_{C,\omega}$ from $K((-)^\at)$ to $K_C$, coming from the universal property of $K$-theory and the adjunction $C\otimes\Ind(-):\Cat^{\mathrm{perf}}\rightleftharpoons \Mod_{C,\omega}:(-)^\at$. 
\end{cons}

For the immediate purpose of exhibiting failures of redshift at higher heights, we do want to point out that the distinction between those two things is irrelevant. We will not use the following fact, which is only a mild variant of \cite[Proposition 4.15]{CMNN}, but it is informative and likely good to know in general:
\begin{lm}\label{lm:agreeK}
    Let $C$ be a commutative rigid $\Sp_{T(n)}$-algebra. In this case, the maps $$K((-)^{\aleph_0})\to K((-)^\at)\to K_C$$ are $T(n+1)$-local equivalences. 

    In particular the map $K(C^\dbl)\to K_C(C)$ is a $T(n+1)$-local equivalence. 
\end{lm}
\begin{proof}
Let $D$ be a compactly generated $C$-module. 
We note that by \Cref{ex:loccompactT(n)}, the argument from \cite[Proposition 4.15]{CMNN} applies to show that the cofiber of the natural inclusion $D^{\aleph_0}\to D^\at$ is $L_{n-1}^f$-linear. 

Together with \cite[Theorem A]{CMNN}, this shows that $K(D^{\aleph_0})\to K(D^\at)$ is a $T(n+1)$-local equivalence, which proves the claim for one of the two maps involved. 

It follows that $T(n+1)\otimes K((-)^\at)$ is a finitary localizing invariant on $\Mod_{C,\omega}$, and since $(-)^\at$ is right adjoint to the functor $\Catperf\to \Mod_{C,\omega}$, it has the same universal property as $T(n+1)\otimes K_C$ among motivic localizing invariants, and they are therefore equivalent. 
\end{proof}
In other words, if we were able only to calculate $K_C(C)$, and prove that \emph{it} does not redshift, we could still condlude that $K(C^\dbl)$ also does not redshift. 

We now recall the main result of \cite{RSW}:
\begin{thm}[\cite{RSW}]
    Let $C$ be a locally rigid\footnote{See footnotes above.} $2$-ring. The functor $$M_C: \Mod_{C,\omega}\to \Mot_C$$ is a Dwyer--Kan localization, and the source has a cofibration structure where the cofibrations are exactly the fully faithful $C$-linear functors. 

    In particular it is essentially surjective. 
\end{thm}
We recall also that the class of maps inverted by $M_C$ consists exactly of the \emph{motivic equivalences}, cf. \cite[Definition 1.4, Definition B.4]{RSW}.

A corollary of the cofibration structure is the following fact about morphisms in $\Mot_C$:
\begin{cor}[{\cite[Remark 2.3, Example 2.4, Proposition 2.8]{RSW}}]
    Let $D,E\in\Mod_{C,\omega}$. For any map $f:M_C(D)\to M_C(E)$ in $\Mot_C$, there exists a fully faithful motivic equivalence $E\to E'$ and a map $D\to E'$ representing $f$ in $\Mot_C$.
\end{cor}
We will need the slightly more precise quantitative version, which is already implicit in \cite{RSW}:
\begin{cor}\label{cor:zigzag}
     Let $D,E\in\Mod_{C,\omega}$. If $D,E$ are $\kappa$-compact in $\Mod_{C,\omega}$ for some regular uncountable cardinal $\kappa$, then for any map $f:M_C(D)\to M_C(E)$ in $\Mot_C$, there exists a fully faithful motivic equivalence $E\to E'$ and a map $D\to E'$ representing $f$ in $\Mot_C$, where $E'$ is itself $\kappa$-compact. 
\end{cor}
\begin{rmk}
    In \cite{RSW}, for simplicity we restricted to those $\kappa$'s we called ``countably closed''. This is no serious loss of generality, especially if we only want to exhibit failures of redshift, but for the more precise results we wish to establish, we do not want to artificially restrict to those when $\kappa=\aleph_1$ does the trick. 
\end{rmk}
\begin{proof}
    By the previous corollary, we may find an $E'$, not necessarily $\kappa$-compact, that does the job. Write $E'$ as a canonical $\kappa$-filtered colimit of $\kappa$-compacts $E_i,i\in I$. By \cite[Lemmas 5.3, 5.8]{RSW}, for some cofinal subcategory  of $I$, the map $E\to E'$ factors through $E_i$ so that $E\to E_i$ is a motivic equivalence. 

Since $D$ is also $\kappa$-compact, the map $D\to E'$ factors through one of those $E_i$'s, as was to be shown. 
\end{proof}
We will also need the following, which is a quantitative version of the construction ``$D\to \Ind(D)^{\aleph_1}$'': 
\begin{lm}\label{lm:disk}
    Let $\kappa$ be an uncountable regular cardinal, and $C$ a locally rigid $2$-ring with $\kappa$-compact unit, and let $M$ be a $\kappa$-compact compactly generated $C$-module. 

    There exists a $\kappa$-compact $C$-module $N$ with trivial motive, and a fully faithful $C$-linear embedding $M\to N$. 
\end{lm}
This lemma has a more general version with base $\V\in\CAlg(\PrL_\st)$.
\begin{proof}
Consider \emph{any} $C$-linear fully faithful embedding $M\to N$ into a motivically trivial $C$-module, as is possible, cf. the paragraph following \cite[Observation B.7]{RSW} (over $\Sp$, or a rigid base, one can simply choose $\hat y: M\to\Ind(M^{\aleph_1})$). 

By \cite[Lemma 5.3]{RSW}, we may write $N$ as a $\kappa$-filtered colimit of $\kappa$-compact, motivically trivial $C$-modules $N_i, i\in I$, where $I$ has $\kappa$-small colimits and $N_\bullet$ preserves them. Taking a pullback $M_i :=M\times_N N_i$, we get $M\simeq \colim_i M_i$. Since $M$ is $\kappa$-compact, \cite[Lemma 5.8]{RSW} guarantees that for some $i$, $M\simeq M_i$, so that $M\simeq M_i\to N_i$ is the desired fully faithful embedding. 
\end{proof}

We conclude this section by introducing a particular localizing invariant, $C$-linear Hochschild homology, as well as the Dennis trace map, which will play a crucial role in the rest of the article: 
\begin{defn}\label{defn:HH}
    Let $C$ be a locally rigid $2$-ring. By \cite[Theorem 1.35]{Dbl} and \cite[Proposition 4.17]{LocRig}, compactly generated $C$-modules are dualizable over $C$ and compact-preserving $C$-linear maps are $C$-linear left adjoints, hence by \cite[Definition 2.11, Theorem 2.14]{HSS}  we obtain a symmetric monoidal dimension functor, which we call $$\HH(-/C): \Mod_{C,\omega}\to C^{BS^1}.$$

    This is a finitary localizing invariant and hence descends to $\Mot_C$. 

We deduce a unique symmetric monoidal, $S^1$-equivariant natural transformation $$\one_C\otimes K_C\to \HH(-/C),$$ and hence also $\one_C\otimes K((-)^\at)\to \HH(-/C)$, or equivalently, $K_C\to \udl{\HH(-/C)}$, resp. $K\to \udl{\HH(-/C)}$. We call either of these the Dennis trace map. 
\end{defn}
\begin{nota}
    In accordance with \Cref{nota:smallrigid} and the discussion following it, we also have motives for small rigid $2$-rings. If $C$ is a small rigid $2$-ring, we will also write $\Mot_C$ and $M_C$ for what would be $\Mot_{\Ind(C)}, M_{\Ind(C)}$ in the previous notation, where again, we hope no confusion will arise. 
\end{nota}
\subsection{Dimension morphisms and the generalized 1$D$ cobordism hypothesis}\label{section:dimensions}
In this section, we discuss three things: first, \emph{symmetric monoidal dimensions} of dualizable objects in symmetric monoidal categories, second, trace maps from Hochschild homology of rigid categories to endomorphisms of their units,  and finally Barkan and Steinebrunner's formula for the free rigid category on a given category.  

We begin by recalling the $1$-dimensional cobordism hypothesis. Recall that this was conjectured in some form by Baez and Dolan \cite{baezdolan} in all dimensions, and a sketch of proof in that generality was given by Lurie \cite{luriecob}. In dimension $1$, a full proof was given by Harpaz \cite{harpazcob}, and Barkan--Steinebrunner \cite{Jan} give a new proof. The following is an informal definition of the $1$D cobordism category:
\begin{defn}
The $1$-dimensional oriented cobordism category $\Cob^{\mathrm{1d,or}}$ is the \category whose objects are oriented closed $0$-dimensional manifolds, i.e. finite sets with a sign $\pm$ on each element, and whose morphisms from $M$ to $N$ are oriented $1$-dimensional compact manifolds with boundary $W$ with an oriented identification $\partial W\cong \overline{M}\coprod N$.
\end{defn}
Here, $\overline{M}$ is the manifold $M$ with the reverse orientation. 
With a more precise version of this definition, the following is a theorem:
\begin{thm}[\cite{harpazcob},\cite{luriecob}]
The free symmetric monoidal \category on a dualizable object is the category $\Cob^{\mathrm{1d,or}}$ of 1-dimensional oriented cobordisms between oriented 0-manifolds.
\end{thm}
In particular, we have: 
\begin{obs}\label{obs:End1Cob}
    In the free symmetric monoidal category on a dualizable object, the endomorphism commutative monoid of the unit is the free commutative monoid on the anima $BS^1$, generated by the dimension of the universal dualizable object.
\end{obs} 
This result allows us in particular to construct $S^1$-action on dimensions:
\begin{cons}
    Let $C$ be a symmetric monoidal category and $x\in C$ a dualizable object. It induces a symmetric monoidal functor $\Cob\to C$ and by looking at endomorphisms of the unit, a map $BS^1\to \End(\one_C)$ which we call $$\Dim(x):BS^1\to \End(\one_C).$$ 
    
    The evaluation at a point is equivalent to the usual symmetric monoidal dimension of $x$, which we denote by $\dim(x)$. 
    
    More coherently, we have a map $\Dim : (C^\dbl)^\simeq \to \End(\one_C)^{BS^1}$ whose value at $x$ is $\Dim(x)$. 
\end{cons}
\begin{ex}
    Under the identification of (topological) Hochschild homology with the symmetric monoidal dimension of the module category, this $S^1$-action is the usual $S^1$-action on Hochschild homology.
\end{ex}
\begin{ex}\label[ex]{ex:dim(1)}
    If $C$ is a $1$-category, $\End(\one_C)$ is a set, and so $\Dim(x)$ factors through the map $BS^1\to \pt$. In particular, this is the case for $x=\one_C$ in \emph{any} symmetric monoidal category since the initial symmetric monoidal category is a $1$-category (namely, the point). 
\end{ex}
We now recall the following: 
\begin{lm}
    Suppose $C$ is stably symmetric monoidal and $x,y,z\in C$ are dualizable. In this case, any cofiber sequence $x\to y\to z$ provides an equivalence $\Dim(y)\simeq \Dim(x)+\Dim(y)$, as maps $BS^1\to \End(\one_C)$. 
\end{lm}
\begin{proof}
    See \cite[Corollary 3.2.2]{keidarragimov}. 
\end{proof}
This result in particular suggests that the dimension map $\Dim: (C^\dbl)^\simeq \to \End(\one_C)^{BS^1}$ factors through $K(C^\dbl)$ - this is in fact how Keidar--Ragimov prove the above-cited result. 

Essentially, any decent factorization works as follows: find an $S^1$-equivariant map $$\THH(C^\dbl)\to \End(\one_C)$$ implementing more generally traces of endomorphisms in $C^\dbl$, and use the Dennis trace to get an $S^1$-equivariant map $K(C^\dbl)\to \THH(C^\dbl)\to \End(\one_C)$, and then by adjunction find the desired map $K(C^\dbl)\to \End(\one_C)^{BS^1}$. \textit{A priori}, the main freedom we have is in the $S^1$-equivariant map $\THH(C^\dbl)\to\End(\one_C)$.

There are several approaches to constructing this map, but there is a certain rigidity\footnote{Pun intended.} to our setting, coming from a sufficient amount of naturality, that actually guarantees uniqueness of this map, if we put it in the right context. To state this naturality, let us introduce some definitions. 
\newcommand{\W}{\mathcal{W}}
\begin{defn}
Let $\V\in\CAlg(\PrL)$ and let $\W$ be a commutative $\V$-algebra. $\W$ is called \emph{atomic rigid} over $\V$ if it is generated under colimits and tensors over $\V$ by dualizable objects, and if all dualizables in $\W$ are atomic over $\V$ (equivalently, the unit of $\W$ is atomic). 

    Consider the functor $\CAlg(\PrL)\to \PrL$ defined by sending $\mathcal V$ to the full subcategory of $\CAlg(\PrL)_{\V/}$ spanned by atomic rigid commutative $\V$-algebras. We let $\CAlg^\rig$ denote its (cocartesian) unstraightening. 

    Consider the forgetful functor $\CAlg(\PrL)\to \PrL$, and let $\CAlg(\PrL)_*$ denote its unstraightening. 
\end{defn}
\begin{rmk}
    Everything we say below would also work with the notion of ``rigid'' from \cite{LocRig} in place of ``atomic rigid'' from the above definition. Since we have chosen to remain within the atomically generated world for this paper, we phrase things in this generality.
\end{rmk}
\begin{cons}
By naturality of $\HH(-/-)$ in both variables, Hochschild homology assembles into a morphism of cocartesian fibrations over $\CAlg(\PrL)$ of the form $\CAlg^\rig\to \CAlg(\PrL)_*$, and so does $\End(\one_{-})$. 
\end{cons}
\begin{thm}\label{thm:uniquetr}
    The anima of natural transformations $\HH(-/-)\to \End(\one_{-})$ as functors $\CAlg^\rig\to \CAlg(\PrL)_*$ over $\CAlg(\PrL)$ is contractible. 

    In particular, the same is true for the anima of $S^1$-equivariant natural transformations. 
\end{thm}
In fact, this does not use a lot about the functors in question. The key point is to use naturality and the arrow $(C,D)\to (D,D)$ and the fact that $(\End(\one_D), C)\to (\one_D, D)$ is a \emph{cartesian} morphism lying over $C\to D$. The following is the abstract situation we are in: 
\begin{lm}\label{lm:axiomunique}
    Let $p_0,p_1: E_0, E_1\to B$ be two coCartesian fibrations, and let $q: E_0\to B$ be a second functor, with a right adjoint $\sigma : B\to E_0$.

    Suppose $B$ has an initial object $\emptyset$. 

    Let $f,g: E_0\to E_1$ be two functors over $B$ and suppose $g$ sends the components of the unit map $\id\to \sigma q$ to $p_1$-cartesian morphisms in $E_1$. 

    Finally, suppose that $f\sigma$ sends the unique map $\emptyset\to b$ to a coCartesian morphism for every $b\in B$, and that $\Map_{E_1}(f\sigma\emptyset, g\sigma\emptyset)$ is contractible. 

    In this case, $\Map_{/B}(f,g)$ is contractible. 
\end{lm}
The long list of assumptions is meant to axiomatize the situation of the theorem:
\begin{proof}[Proof of \Cref{thm:uniquetr}]
We verify the assumptions of the lemma, where we take $B=\CAlg(\PrL), E_0= \CAlg^\rig, E_1=\CAlg(\PrL)_*$, $f=HH(-/-), g=\End(\one_{-})$, and finally $q$ is the functor $(C,D)\mapsto D$. 

    The right adjoint $\sigma$ to $q$ is $D\mapsto (D,D)$, as is easily verified, and the unit map is the canonical map $(C,D)\to (D,D)$. The functor $g$ sends this to $(C,\End(\one_D))\to (D,\one_D)$ which is cartesian essentially by definition. 

    The initial object of $B$ is $\An$, and in this case $f\sigma = (C\mapsto (C,\one_C))$ indeed sends each map $\An\to C$ to the cocartesian edge $(\An, \pt)\to (C,\one_C)$ lying over $\An\to C$.

    Finally, $\Map_{\CAlg(\PrL)_*}((\An,\pt),(\An,\pt))$ is contractible. Thus, all the assumptions are satisfied, and the result indeed follows from \Cref{lm:axiomunique}.
\end{proof}

We now prove the lemma. 
\begin{proof}[Proof of \Cref{lm:axiomunique}]

Given functors $h,k: E_0\to E_1$ and a natural transformation $\theta: p_1\circ h \to p_1\circ k$, we let $\Map_{/\theta}(h,k)$ denote the pullback $$\Map_{\Fun(E_0,E_1)}(h,k)\times_{\Map_{\Fun(E_0,B)}(p_1\circ h,p_1\circ k)} \{\theta\}$$ 
For example, if $h,k$ are functors over $B$ and $\theta$ is the identity transformation from $p_0$ to itself, $\Map_{/\id}(h,k)$ is simply the mapping anima in $\Fun_{/B}(E_0,E_1)$. 

Consider thus the map induced by postcomposition: $$\Map_{/\id}(f,g) \to \Map_{/\eta}(f, g\sigma q)$$ where $\eta$ is $p_0$ applied to the unit map $\id\to \sigma q$.

The assumption that $g$ sends the components of the unit map $\id\to\sigma q$ to cartesian edges implies that this map is an equivalence. 

Now, $\sigma$ is right adjoint to $q$ and so precomposition by $q$ is right adjoint to precomposition by $\sigma$. It follows that the latter mapping anima is equivalent to $\Map_{/\id}(f\sigma, g\sigma)$ where now $\id$ is the identity transformation of $p_0\sigma$. 

Now we use the natural transformation $\empty\to b$ on $B$ to obtain, by precomposition, a map $$\Map_{/\id}(f\sigma, g\sigma) \to \Map_{/!}(f\sigma\emptyset, g\sigma)$$
where $!$ denotes the natural map $p_0\sigma\emptyset\to p_0\sigma$ induced by $\emptyset\to \sigma$. Our assumption that $f\sigma$ sends $\emptyset\to b$ to a coCartesian edge for al $b$ implies that this map is, in turn, an equivalence. 

Again by adjunction, we have an equivalence $$\Map_{/!}(f\sigma\empty, g\sigma)\simeq \Map_{E_1}(f\sigma\empty,g\sigma\empty)$$
Our assumption is that the latter is contractible. Pasting together all the equivalences so far, we find the desired claim. 
\end{proof}
\begin{rmk}
    Concretely, in our context, the proof unwinds informally to the following argument: Fix a natural transformation $\HH(-/-)\to \End(\one_{-})$, which we want to understand, and let $C\in\CAlg(\PrL)$ and $D$ an atomic rigid  commutative $C$-algebra. We want to understand $\HH(D/C)\to \End(\one_D)$. Since $\End(\one_D)$ is definitionally the right adjoint of $C\to D$ applied to $\one_D$, it suffices really to understand the map $(\HH(D/C),C)\to (\one_D,D)$ in $\CAlg(\PrL)_*$. 

    Now, by naturality applied to the map $(D,C)\to (D,D)$ in $\CAlg^\rig$, we have a commutative square of the form:
\[\begin{tikzcd}
	{(\HH(D/C),C)} & {(\End(\one_D),C)} \\
	{(\HH(D/D),D)} & {(\one_D,D)}
	\arrow[from=1-1, to=1-2]
	\arrow[from=1-1, to=2-1]
	\arrow[from=1-2, to=2-2]
	\arrow[from=2-1, to=2-2]
\end{tikzcd}\]
telling us exactly that this map can be understood as the map $\HH(D/C)\to \HH(D/D)$ \emph{coming from functoriality of $\HH$}, followed by a map $\HH(D/D)\to \one_D$. 

Now $\HH(D/D)=\one_D$, and there is a unique transformation $\one_D\to \one_D$ natural in $D$ since the initial $D$ is $\An$.

Thus, the transformation is entirely determined from functoriality. 
\end{rmk}
\begin{rmk}\label{rmk:desctr}
    In the above discussion, we had the abstract map $(\HH(D/C),C)\to (\HH(D/D),D)$ coming from functoriality ``across fibers'' to understand. This may seem nontrivial, but once we observe that $\HH$ has an extra property not axiomatized in \Cref{lm:axiomunique}, it becomes much clearer: $\HH$ preserves coCartesian edges, i.e. $\HH(D\otimes_C E/E)\simeq \eta \HH(D/C)$ where $\eta: C\to E$ is any map in $\CAlg(\PrL)$.

    Thus, the map $(\HH(D/C),C)\to (\HH(D/D),D)$ coming from the (non-cocartesian) map $(D,C)\to (D,D)$ can also be encoded using the coCartesian factorization $(D,C)\to (D\otimes_C D,D)\to (D,D)$ where now $D\otimes_C D\to D$ is the multiplication map $\mu$. Thus, another way to understand the transformation $\HH(D/C)\to \End(\one_D)$ is as follows: we instead understand its mate $\one_D\otimes\HH(D/C)\to \one_D$ which in turn is given as: $$\one_D\otimes\HH(D/C)\simeq \HH(D\otimes_C D/D)\xrightarrow{\HH(\mu/D)}\HH(D/D)= \one_D$$

    This is exactly the description mentioned in \cite[Remark 2.12]{meadd} and in the proof of \cite[Lemma 3.2.1]{keidarragimov}. In turn, this description makes it apparent that the trace map is a map of commutative algebras, since $\mu$ clearly is! 
\end{rmk}
\begin{cor}\label{cor:trconstr}
    The unique natural transformation $\HH(D/C)\to \End(\one_D)$ on $\CAlg^\rig$ is the mate of the following map:
    $$\one_D\otimes\HH(D/C)\simeq \HH(D\otimes_C D/D)\xrightarrow{\HH(\mu/D)}\HH(D/D)= \one_D$$
    In particular, it is an $S^1$-equivariant map of commutative algebras. 
\end{cor}
\begin{proof}
    See the discussion in \Cref{rmk:desctr}. 
    
    Alternatively, for the fact that it is a map of commutative algebras with $S^1$-action, use \Cref{lm:axiomunique} with $B=\CAlg(\PrL), E_0=\CAlg^\rig$ and $E_1$ the coCartesian unstraightening of the functor $\CAlg(\PrL)\to \PrL$ given by $\V\mapsto \CAlg(\V)^{BS^1}$, and the verification of the assumptions is exactly the same as in the proof of \Cref{thm:uniquetr}. 

    This produces a unique map of commutative algebras with $S^1$-action $\HH(-/-)\to \End(\one_{-})$ which must agree with the map from \Cref{thm:uniquetr} when forgetting the extra structure, since the latter is also unique. 
\end{proof}

\begin{nota}
    We call the unique map $\HH(-/-)\to \End(\one_{-})$ on $\CAlg^\rig$ the ``trace map'' (not to be confused with the Dennis trace map).
\end{nota}

\begin{obs}
    Forgetting the commutative algebra structure and remembering only the $S^1$-action, the functor $\HH$ on $\CAlg^\rig$ factors as $\CAlg^\rig\to \CAlg^\at\to \CAlg(\PrL)_*$ where $\CAlg^\at$ is the coCartesian unstraightening of the functor $\V\mapsto \Mod_\V(\PrL)^\at$ defined on $\CAlg(\PrL)$, where $\Mod_\V(\PrL)^\at$ is the category of atomically generated $\V$-modules and atomic preserving $\V$-linear functors, as in \cite[Corollary 3.15]{Dbl}.

    The functor $\CAlg^\rig\to \CAlg^\at$ has a left adjoint given by $\Frrig_{-}$ by combining the main result of \cite{Frrig} with the theory of relative adjunctions, cf. \cite[Proposition 7.3.2.11]{HA} and so ($S^1$-equivariant) natural transformations $\HH(-/-)\to \End(\one_{-})$ on $\CAlg^\rig$ are equivalent to ($S^1$-equivariant) natural transformations $\HH(-/-)\to \End(\one_{\Frrig_{-}(-)})$ on $\CAlg^\at$. In particular, the anima of the latter kind is also contractible. 
\end{obs}

Barkan and Steinebrunner proved the following description of free rigid categories in the enriched context, in which the (now proved to be unique) transformation $\HH(-/-)\to \End(\one_{\Frrig_{-}(-)})$ plays a crucial role:
\begin{thm}[Barkan--Steinebrunner, \cite{Jan}]\label{thm:Jan}
    Let $\V\in\CAlg(\PrL)$, and let $M$ be a $\V$-enriched category. The free rigid symmetric monoidal $\V$-enriched category on $M$ has: 
    \begin{enumerate}
        \item The endomorphism ring of the unit $\one$ is the free commutative algebra in $\V$ on $\HH(M/\V)_{hS^1}$; 
        \item For every sequence of objects $x_1,...,x_n, y_1,...,y_m$ in $M$, an object $\bigotimes_i x_i \otimes \bigotimes_j y_j^\vee$; 
        \item For every sequence of objects $x_1,...,x_n, y_1,...,y_m$ as above, the mapping object from the unit $$\hom(\one, \bigotimes_i x_i\otimes\otimes_j y_j^\vee) $$  is a coproduct indexed by cobordisms $\emptyset\to n^+ \coprod m^-$ of tensor products of objects $\hom(y_j,x_i)$ whenever $i,j$ are the endpoints of one connected component of the cobordims in question, tensored up to $\End(\one)$.  
    \end{enumerate}
\end{thm}
The relation to our story is that stable categories can be viewed as special cases of spectrally enriched categories, cf. \cite[Theorem 1.11]{heine} (with a precursor in \cite[Example 7.4.14]{gepnerhaugseng}). More generally, if $C$ is a locally rigid $2$-ring, compactly generated $C$-modules can be viewed as special cases of $C$-enriched categories. This gives us: 
\begin{cor}\label{cor:unitinFrrig}
    Let $C\in\CAlg(\PrL)$ and let $M$ be an atomically generated $C$-module. The endomorphism ring of the unit in $\Frrig_C(M)$ is the free commutative algebra in $C$ on $\HH(M/C)_{hS^1}$. 

    The morphism objects in $C$ from the unit in $\Frrig_C(M)$ to tensor products of atomic objects in $M$ and their duals are computed as in \Cref{thm:Jan}. 
\end{cor}
We also refer to \cite[Corollary 4.2]{Frrig} for a proof that does not go through enriched categories and instead works directly at the level of $C$-modules (and in particular generalizes to dualizable $C$-modules which need not be atomically generated).

A special case of this is the following observation, also a consequence of the usual $1$D cobordism hypothesis: 
\begin{obs}\label{obs:QEnd1Cob}
Let $C\in\CAlg(\PrL)$. The endomorphism ring of the unit in $\Frrig_C(C)$ is the free commutative algebra in $C$ on $BS^1$.
\end{obs}
\begin{nota}
    For an object $M\in C$, we let $\one_C\{M\}$ denote the free commutative algebra in $C$ on $M$. When $X$ is an anima, we shorten $\one_C\{\one_C[X]\}$ to $\one_C\{X\}$, hoping no confusion will arise.
\end{nota}

We note that in \cite{Jan} and \cite{Frrig}, the explicit trace map constructed in \Cref{cor:trconstr} is not explicitly used, and the identification of $\End(\one_{\Frrig_C(M)})$ in terms of $\HH(M/C)_{hS^1}$ is a construction rather than an observation.
In fact, we can use these works to provide an alternative construction: 
\begin{cons}
    Let $C\in\CAlg(\PrL)$, and let $D$ be an atomically generated $C$-module. We define the \emph{canonical map} $\can : \HH(D/C)_{hS^1}\to \End(\one_{\Frrig_C(D)})$ to be the one arising from \Cref{cor:unitinFrrig}. Since the identification in the latter is natural in $D$, so is this map. 

    When $D$ is an atomic rigid commutative $C$-algebra, we have a commutative $C$-algebra map $\Frrig_C(D)\to D$, and we use this to define the \emph{trace map} $\tr:\HH(D/C)_{hS^1}\to \End(\one_D)$ to be the composite of $\can$ with the induced map $\End(\one_{\Frrig_C(D)})\to\End(\one_D)$. This is a natural transformation of functors on $\CAlg^\rig$, and hence agrees with the trace map discussed earlier by \Cref{thm:uniquetr}. In particular, the corresponding $S^1$-equivariant map $\HH(D/C)\to \End(\one_D)$ is (canonically) a commutative algebra map. 
\end{cons}

Putting this together with the previous discussion, we get the following:
\begin{cons}\label{cons:Dim}
For $E\in \TwoRing^\locrig$, we define the Dimension morphism $$\Dim: K(E^\dbl)\to \End(\one_E)^{BS^1}$$ as the mate of the map $$\one_E\otimes K(E^\dbl)[BS^1]\simeq \one_E\otimes K(E^\dbl)_{hS^1}\to \one_E\otimes K_E(E)_{hS^1} \to \HH(E/E)_{hS^1}\xrightarrow{\tr}\one_E$$ where the map $\one_E\otimes K_E(E)\to \HH(E/E)$ is the $S^1$-equivariant Dennis trace introduced in \Cref{defn:HH}.

We sometimes also abusively denote by $\Dim$ the mate $K(E^\dbl)[BS^1]\to \End(\one_E)$, and similarly for the version with $K_E(E)$ in place of $K(E^\dbl)$.
\end{cons}
The following is an elementary observation which will come in handy later: 
\begin{obs}\label{obs:HHtr}
    Let $C$ be a locally rigid $2$-ring and $D$ a compactly generated $C$-module. Let $\iota_D : D\to \Frrig_C(D)$ be the canonical map. There is a natural commutative triangle in $C$:
    
\[\begin{tikzcd}
	{\HH(D/C)_{hS^1}} & {\HH(\Frrig_C(D))_{hS^1}} \\
	{\End(\one_{\Frrig_C(D)})}
	\arrow["{\HH(\iota_D)_{hS^1}}", from=1-1, to=1-2]
	\arrow["\can"', from=1-1, to=2-1]
	\arrow["\tr", from=1-2, to=2-1]
\end{tikzcd}\]
    
\end{obs}
\begin{proof}
    By naturality of $\can$, we have a natural commutative square as follows: 
    \[\begin{tikzcd}
	{\HH(D/C)_{hS^1}} & {\HH(\Frrig_C(D))_{hS^1}} \\
	{\End(\one_{\Frrig_C(D)})} & {\End(\one_{\Frrig_C(\Frrig_C(D))})}
	\arrow["{\HH(\iota)_{hS^1}}", from=1-1, to=1-2]
	\arrow["\can"', from=1-1, to=2-1]
	\arrow["\can", from=1-2, to=2-2]
	\arrow["{}"', from=2-1, to=2-2]
\end{tikzcd}\]

Using the multiplication $\Frrig_C\Frrig_C(D)\to \Frrig_C(D)$ and the fact that the composite $\Frrig_C(D)\xrightarrow{\Frrig_C(\iota)}\Frrig_C\Frrig_C(D)\to \Frrig_C(D)$ is naturally equivalent to the identity, we get the desired triangle. 
\end{proof}
The following corollary is the main way in which this observation will be applied: 
\begin{cor}\label{cor:commutativediagram}
    Let $C$ be a locally rigid $2$-ring and let $D$ be a compactly generated $C$-module. For any $C$-linear map $f:D\to C$, inducing a map $\tilde f: \Frrig_C(D)\to C$, the following diagram commutes: 
    \[\begin{tikzcd}
	{K(D^\at)[BS^1]} & {K(C^\dbl)[BS^1]} \\
	{\udl{\HH(D/C)_{hS^1}}} \\
	{\End(\one_{\Frrig_C(D)})} & {\End(\one_C)}
	\arrow[from=1-1, to=1-2]
	\arrow[from=1-1, to=2-1]
	\arrow["\Dim", from=1-2, to=3-2]
	\arrow[from=2-1, to=3-1]
	\arrow[from=3-1, to=3-2]
\end{tikzcd}\]
The same is true replacing $K(D^\at)$, resp. $K(C^\dbl)$ with $K_C(D)$, resp. $K_C(C)$. 
\end{cor}
\begin{proof}
    We complete it to a bigger diagram: 
    \[\begin{tikzcd}
	{K(D^\at)[BS^1]} & {K(\Frrig_C(D)^\dbl)[BS^1]} & {K(C^\dbl)[BS^1]} \\
	{\udl{\HH(D/C)_{hS^1}}} & {\udl{\HH(\Frrig_C(D/C))_{hS^1}}} & {\udl{\HH(C/C)_{hS^1}}} \\
	{\End(\one_{\Frrig_C(D)})} && {\End(\one_C)}
	\arrow[from=1-1, to=1-2]
	\arrow[from=1-1, to=2-1]
	\arrow[from=1-2, to=1-3]
	\arrow[from=1-2, to=2-2]
	\arrow[from=1-3, to=2-3]
	\arrow[from=2-1, to=2-2]
	\arrow[from=2-1, to=3-1]
	\arrow[from=2-2, to=2-3]
	\arrow[from=2-2, to=3-1]
	\arrow[from=2-3, to=3-3]
	\arrow[from=3-1, to=3-3]
\end{tikzcd}\]
The outer diagram is the desired diagram by design. Each of the top square commutes by naturality of the Dennis trace. The bottom left triangle commutes by \Cref{obs:HHtr}, and the bottom right quadrangle commutes by naturality of $\tr$ along symmetric monoidal functors of rigid commutative $C$-algebras. 

The proof with $K_C(D),K_C(C)$ is exactly the same. 
    \end{proof}

\section{$K$-theory and Dimensions}\label{section:main} 
The goal of this section is to prove our main theorem:
\begin{thm}\label{thm:designer}
    Let $C$ be a locally rigid $2$-ring. There exists a rigid extension $C\to D$ such that the two morphisms in the following composite are equivalences: $$K(D^\dbl)\to K_C(D) \to \End(\one_D)^{BS^1}.$$ 
\end{thm}
Note that since $D$ is a rigid extension of $C$, it is nonzero, but more importantly $\End(\one_D)\simeq \End(\one_C)$. Before proving the theorem, we draw our main conclusion from it: 
\begin{cor}\label{cor:noshift}
     Let $C$ be a nonzero rigid $T(n)$-local symmetric monoidal \category. There exists a rigid extension $C\to D$ such that $$L_{T(n+1)}K(D^\dbl)= L_{T(n+1)}K_{\Sp_{T(n)}}(D) = 0.$$
\end{cor}
\begin{proof}
    Let $C\to D$ be as in \Cref{thm:designer}. The ring $\End(\one_C)$ is $T(n)$- and hence $L_n^f$-local, so the same holds for $\End(\one_C)^{BS^1}$ and thus its $T(n+1)$-localization vanishes, as was to be shown. 
\end{proof}
Note that the extensions from \Cref{thm:designer} are obtained via a small object argument, and in particular can be ``very big''. In fact, the theorem itself forces them to have at least $2^{\aleph_0}$ dualizable objects. One may therefore wonder how ``prevalent'' the failure of redshift is for ``small'' rigid $T(n)$-local categories. In fact, since the failure this case is so drastic, one can immediately see that the failure happens already at smaller stages. The proof will also show this, but we briefly digress to explain how to directly deduce it from the above result: 
\begin{cor}\label{cor:smallnoshift}
    Let $C\in \TwoRing^\rig_{T(n)}$ be arbitrary. There exists an $\aleph_1$-compact rigid commutative $C$-algebra $D$ such that $L_{T(n+1)}K(D^\dbl) =0$. 
\end{cor}
\begin{proof}
    Fix a rigid extension $C\to D_0$ satisfying the conclusion of \Cref{cor:noshift}.

    The functor $T(n+1)\otimes K((-)^\dbl)$ commutes with $\aleph_1$-filtered colimits (in fact it commutes with filtered colimits\footnote{Nonetheless, the rest of the argument needs $\aleph_1$, because $\TwoRing^\rig_{\Sp_{T(n)}}$ is not compactly generated, only $\aleph_1$-compactly generated.}). Writing $D_0$ as an $\aleph_1$-filtered colimit of $\aleph_1$-compact rigid commutative $C$-algebras $D_i, i\in I$ and using \cite[Lemma 5.8]{RSW} we find that for some $i_0\in I$, $L_{T(n+1)}K(D_{i_0}^\dbl) = 0$. 
\end{proof}

To prove \Cref{thm:designer}, we implement the strategy described in the introduction, which we recall here.

We will prove that any class in the kernel of $\pi_n\Dim$ can be killed in some rigid extension of $C$, and that any class in $\pi_n \End(\one_C)^{BS^1}$ can be realized in the $K$-theory of some rigid extension of $C$, and conclude by using a small object argument which essentially applies this mechanism inductively until all classes in the kernel have been killed, and all classes in the target have been reached. 

We do this by \emph{categorifying} homotopy classes in the source and the target, and observing that properties of these homotopy classes are reflected in \emph{categorical} properties, thanks to \Cref{thm:Jan}, and more specifically its consequence \Cref{cor:commutativediagram}. The surjectivity part (about realizing classes in the target) uses essentially only these ideas (and the ability to categorify de/suspensions of $K$-theory), while the injectivity part (killing classes in the kernel) really uses the main result of \cite{RSW}, which gives us the ability to categorify $K$-theory classes.

If we are not \emph{too} careful, in the case where $C$ is only locally rigid (as opposed to rigid), we only obtain the result concerning $K_C(D)$, as opposed to $K(D^\dbl)$. As explained in \Cref{lm:agreeK}, in the $T(n)$-local case, this would be sufficient to prove the failure of redshift, \emph{also} for $D^\dbl$ (since its $T(n+1)$-local $K$-theory agrees with the $T(n+1)$-localization of that relative $K$-theory spectrum), but we would like to actually get a calculation of $K(D^\dbl)$, so we will try to be careful.

We start by recalling an additional consequence of \Cref{cor:unitinFrrig}, namely: 
\begin{prop}[{\cite[Corollary 4.11]{Frrig}}]\label{prop:Frrigff}
Let $C$ be a locally rigid $2$-ring, and let $A\to B$ be a fully faithful morphism of compactly generated $C$-modules. The functor $$\Frrig_C(A)\otimes_{\End(\one_{\Frrig_C(A)})}\End(\one_{\Frrig_C(B)})\to\Frrig_C(B)$$ is fully faithful.  

In particular, if furthermore $\HH(A/C)\to \HH(B/C)$ is an equivalence, then $$\Frrig_C(A)\to\Frrig_C(B)$$ is fully faithful. 
\end{prop}
\begin{rmk}
    Since \Cref{cor:unitinFrrig} has a generalization to dualizable $C$-modules, cf. \cite{Frrig}, so does the above proposition. This is the generality in which it is stated in \textit{loc.cit.}. 
\end{rmk}
\begin{proof}
 By \Cref{cor:unitinFrrig}, the second claim clearly follows from the first, so we prove the latter. 

For that one, it suffices to check on generators, and by dualizability, on maps from the unit to generators. There, \Cref{cor:unitinFrrig} proves that the $C$-morphism objects are free $\End(\one_{\Frrig_C(A)})$-modules over coproducts of tensor products of morphism objects in $A$. Since the latter are equivalent to the morphism objects of their image in $B$, the claim follows.   
\end{proof}

\begin{cor}\label{cor:almostself-inj}
Let $C$ be a locally rigid $2$-ring with $\kappa$-compact unit and let $f:C\to C'$ be a fully faithful $C$-motivic equivalence between $C$-modules, where $C'$ is $\kappa$-compact. 

There exists a $\kappa$-compact rigid extension $D$ of $C$ (in the sense of \Cref{defn:ext}) such that the unit map $C\to D$ factors through $C'$. Equivalently, such that $D\otimes_C f$ admits a $D$-linear retraction.
\end{cor}
\begin{proof}
  The fact that these two statements are equivalent is easy: if we have a factorization $C\to C'\to D$, then upon tensoring with $D$ and using the multiplication map $D\otimes_C D\to D$, we find the desired retraction; and conversely given a retraction, we can simply compose it with the unit map $C'\to D\otimes_C C'$  to get the desired factorization. 

  We now prove the factorization claim. 

    Consider $\Frrig_C(C')$: it is a $\kappa$-compact rigid commutative $C$-algebra, and by \Cref{prop:Frrigff}, $\Frrig_C(C)\to \Frrig_C(C')$ is fully faithful. 

The pushout $D := \Frrig_C(C')\otimes_{\Frrig_C(C)}C$ along the counit map $\Frrig_C(C)\to C$ is therefore an extension of $C$, by \Cref{cor:fftensor}. 

It is also $\kappa$-compact by inspecting its universal property, and by construction we have that the composite $C\to C'\to \Frrig_C(C')\to D$ is the unit map, hence $D$ is our desired extension. 
\end{proof}

We can now extract the following general consequence, saying essentially that up to relatively small rigid extensions, any map in motives into $M(C)$ can be realized directly by a map of categories: 
\begin{cor}\label{cor:almostrealizingmaps}
    Let $C$ be a locally rigid $2$-ring with $\kappa$-compact unit. Let $f:M_C(E)\to M_C(C)$ be a map in $\Mot_C$, where $E$ is a $\kappa$-compact compactly generated $C$-module for $\kappa$ uncountable regular.
    
    There exists a $\kappa$-compact rigid extension $C\to D$ and a $D$-linear map $\overline{f}: D\otimes_C E\to D$ such that $M_D(\overline f)\simeq D\otimes_C f$. 
\end{cor} 
\begin{proof}
    By \Cref{cor:zigzag}, we can represent $f$ as a zigzag $E\to C'\xleftarrow{\sim}C$ where $C\to C'$ is a fully faithful motivic equivalence and $C'$ is $\kappa$-compact. 

Therefore \Cref{cor:almostself-inj} applies and we find a $\kappa$-compact rigid extension $D$ of $C$ and a retraction $D\otimes_CC'\to D$. Since the composite $D\to D\otimes_CC'\to D$ is the identity, $D\otimes_C C'\to D$ is a motivic inverse to $D\to D\otimes_C C'$. We therefore find that the composite $\overline{f}:D\otimes_C E\to D\otimes_C C'\to D$ has the correct image under $M_D$. 
\end{proof}

We also obtain a similar consequence for the $K$-theory of dualizable objects rather than $K_C$:
\begin{cor}\label{cor:almostrealizingclasses}
Let $C$ be a locally rigid $2$-ring with $\kappa$-compact unit, and let $f:C^\dbl\to C'$ be a fully faithful motivic equivalence of $C^\dbl$-modules, where $C'$ is $\kappa$-compact. 

There exists a $\kappa$-compact rigid extension $D$ of $C$ and a map $C'\to D^\dbl$ such that the composite $C^\dbl\to C'\to D^\dbl$ is the unit map. Equivalently, such that $D^\dbl\otimes_{C^\dbl}f$ admits a retraction.

As a consequence, for any map $M_{C^\dbl}(E)\to M_{C^\dbl}(C^\dbl)$ in $\Mot_{C^\dbl}$ where $E$ is a $\kappa$-compact $C^\dbl$-module, there exists a $\kappa$-compact rigid extension $D$ of $C$ such that upon basechange to $D^\dbl$, there exists a map $\overline{f}: D^\dbl\otimes_{C^\dbl}E\to D^\dbl$ such that the map $M_{D^\dbl}(\overline{f})\simeq D^\dbl\otimes_{C^\dbl}f$ in $\Mot_{D^\dbl}$. 
\end{cor}

\begin{proof}
     The ``As a consequence'' part follows exactly as above; and the ``Equivalently'' statement has also been explained earlier, so we focus on the first statement.
     
     So consider a fully faithful motivic equivalence $C^\dbl \to C'$. Taking Ind-completions and basechanging along $\Ind(C^\dbl)\to C$, we obtain a map of $C$-modules $C\to \tilde{C}'$ which is now a fully faithful $C$-linear motivic equivalence. 

     By \Cref{cor:almostself-inj}, we find a $\kappa$-compact rigid extension $D$ of $C$ with a map $\tilde{C}'\to D$ factoring the unit map $C\to D$. 

     Note that by \cite[Lemma 4.13]{LocRig}, the map $C'\to \tilde{C}'$ has values in $(\tilde{C}')^\at$ and $\tilde{C}'\to D$ is $C$-internally left adjoint and hence preserves atomics. Since $D$ is rigid over $C$, atomics are the same as dualizables, and so the composite $C'\to \tilde{C}'\to D$ lands in $D^\dbl$. Since the composite $C^\dbl\to C'\to D^\dbl\to D$ is the (restriction of) the unit map, and $D^\dbl\to D$ is fully faithful, $C^\dbl\to C'\to D^\dbl$ is the unit map, as was to be shown. 
\end{proof}
We now have almost all tools at hand to prove our main result - in fact, we could already prove the rigid version (e.g. the rational version, which suffices to disprove redshift at height $0$, or the $\mathbb F_p$-version, to answer \cite[Conjecture 5, Question 16]{levycat}), but for the locally rigid version (relevant for redshift at higher heights), the subtleties pointed out in \Cref{section:Mot} require us to introduce a new notion:
\begin{defn}
    Let $C$ be a locally rigid $2$-ring and $D$ be a compactly generated $C$-module. A homotopy class in $\pi_n \map_{\Mot_C}(M_C(C),M_C(D))$ is called atomic if it is in the image of $\pi_n K(D^\at)$ under the map $K(D^\at)\to \map_{\Mot_C}(M_C(C),M_C(D))$ from \Cref{cons:compK}. 

    If $M\in \Mot_C$ is a $C$-linear motive, a homotopy class in $\pi_n\map_{\Mot_C}(M_C(C),M)$ is called $\kappa$-atomic if there exists a $\kappa$-compact compactly generated $C$-module $D$ with $M_C(D)\simeq M$ such that the homotopy class in question is atomic in the above sense.
\end{defn}
\begin{rmk}
    If $C$ is rigid, these conditions are automatic. 
\end{rmk}
Recall that the map under consideration is typically not an equivalence, so one may worry about the existence of a sufficient supply of ($\kappa$-)atomic homotopy classes. We provide the following examples to show that this is actually a reasonable notion. The first is an obvious example: 
\begin{ex}\label{ex:atomicunit}
    The element $1\in \End_{\Mot_C}(M_C(C))$ is atomic, since the unit $\one_C\in C$ is atomic by definition.
\end{ex}
This second example is more important and will be used later on - it allows us to create new atomic homotopy classes out of old ones:
\begin{ex}\label{ex:atomicshift}
Let $\kappa$ be an uncountable regular cardinal such that $\one_C$ is $\kappa$-compact. 
    Let $M\in\Mot_C$ be a motive, and $x\in \pi_n\map_{\Mot_C}(M_C(C),M)$ be a $\kappa$-atomic homotopy class. In this case, $x$, viewed as a class in $\pi_{n+k}\map_{\Mot_C}(M_C(C), \Sigma^k M)$ is also $\kappa$-atomic.
\end{ex}
\begin{proof}
Let $D$ be a $\kappa$-compact compactly generated $C$-module with $M_C(D)\simeq M$ such that $x$ is in the image of $K(D^\at)$, say the image of $\overline{x}\in \pi_n K(D^\at)$. 

Consider now a $\kappa$-compact ordinary compactly generated stable category $E$ with $\Sp$-linear motive equivalent to $\Sigma^k M_{\Sp}(\Sp)$, which exists by \cite[Proposition 5.1]{RSW}\footnote{If $\kappa\geq 2^{\aleph_0}$, then we can simply pick $E=\mathrm{Calk}_{\mathrm{cont}}(\Sp)^{\otimes k}$, cf. \cite[Definition 1.60]{efimovI}, the given citation is only to guarantee the $\kappa$-compactness if $\kappa$ is arbitrary.}. Then $K(E^{\aleph_0} \otimes D^\at)\simeq \Sigma^k K(D^\at)$, and $M_C(E\otimes D)\simeq \Sigma^k M_C(D)$. 

Finally, the map $E^{\aleph_0}\otimes D^\at\to E\otimes D$ lands in $(E\otimes D)^\at$ by \cite[Lemma 4.13]{LocRig} together with the fact that external tensor products of atomics are atomics (since tensor products of internal left adjoints are internal left adjoints). It follows that the image of the class $\overline{x}\in \pi_nK(D^\at)\cong \pi_{n+k}(\Sigma^k K(D^\at))\cong \pi_{n+k}(K(E^{\aleph_0}\otimes D^\at))$ under $K(E^{\aleph_0}\otimes D^\at)\to K((E\otimes D)^\at)$ is a lift as desired. 
\end{proof}
\begin{ex}\label{ex:atomiccompos} Let $\kappa$ be an uncountable regular cardinal such that $\one_C$ is $\kappa$-compact. 

    Let $M\in \Mot_C$ and let $x\in \pi_n\map_{\Mot_C}(M_C(C), M)$ be a $\kappa$-atomic homotopy class. For any map $M\to M_C(C)$, there exists a $\kappa$-compact rigid extension $D$ of $C$ such that the image of $x$ in $\pi_n\End_{\Mot_C}(M_C(C),M_C(C))$ is sent to an atomic class in $\pi_n\End_{\Mot_D}(M_D(D),M_D(D))$
\end{ex}
\begin{proof}
    Fix a $\kappa$-compact compactly generated $C$-module $E$ with $M_C(E)\simeq M$ such that $x$ lifts to $K(E^\at)$. By \Cref{cor:almostrealizingmaps}, there exists a $\kappa$-compact rigid extension $D$ of $C$ such that, after extension along $\Mot_C\to \Mot_D$ the map $M_C(E)\simeq M\to M_C(C)$ can be realized as a map $D\otimes_CE\to D$ of $D$-modules. 
    
    Since this sends atomics to atomics, we have a commutative diagram as follows, proving the result: 
 \[\begin{tikzcd}
	{K(E^\at)} & {K((D\otimes_CE)^\at)} & {K(D^\dbl)} \\
	{\map_{\Mot_C}(M_C(C),M_C(E))} & {\map_{\Mot_D}(M_D(D),M_D(D\otimes_CE))} & {\map_{\Mot_D}(M_D(D),M_D(D))}
	\arrow[from=1-1, to=1-2]
	\arrow[from=1-1, to=2-1]
	\arrow[from=1-2, to=1-3]
	\arrow[from=1-2, to=2-2]
	\arrow[from=1-3, to=2-3]
	\arrow[from=2-1, to=2-2]
	\arrow[from=2-2, to=2-3]
\end{tikzcd}\]
\end{proof}

We can now start proving results about the Dimension morphism itself. Namely, we will prove four results, two results about each of the two maps $K((-)^\dbl) \to \End(\one)^{BS^1}$, $K_C(-)\to \End(\one)^{BS^1}$. The results are either  surjectivity results of the form ``any class in the target is realized in the source up to some sufficiently large extension'' or injectivity results of the form ``any homotopy class that vanishes under this map already vanishes in the source up to a small extension''. 

We fix some notation for the next few lemmas:
\begin{nota}\label{nota:std}
For the next four lemmas, $C$ is a locally rigid $2$-ring, and $\kappa$ is an uncountable regular cardinal such that $\one_C$ is $\kappa$-compact.     
\end{nota}

\begin{lm}[Atomic Surjectivity]\label{lm:atomicsurj}
Let $C$ be as in \Cref{nota:std}, and let $x\in \pi_n\Map(BS^1,\End(\one_C))$. 

There exists a $\kappa$-compact rigid extension $D$ of $C$ and class $\tilde{x}\in \pi_nK(D^\dbl)$ which maps to $x$ under the Dimension morphism $K(D^\dbl)\to \End(\one_D)^{BS^1}\simeq \End(\one_C)^{BS^1}$. 
\end{lm}
\begin{proof}
     Let $x\in \pi_n\map(BS^1,\End(\one_C))$. Let $D_0$ be an $\kappa$-compact compactly generated $C$-module such that $M_C(D_0)=\Sigma^n M_C(C)$, and such that the canonical class in $\pi_n \map_{\Mot_C}(M_C(C),M_C(D_0))$ is atomic - the existence of such a $D_0$ is guaranteed by \Cref{ex:atomicunit} and \Cref{ex:atomicshift}. 

Consider now the map in $\CAlg(C)$ from the free commutative algebra on $\HH(D_0/C)_{hS^1} = \Sigma^n\one_C[BS^1]$ to $\one_C$ classified by $x$. Using \Cref{cor:unitinFrrig}, this free commutative algebra is the endomorphism algebra of the unit in $\Frrig_C(D_0)$, so we can basechange $\Frrig_C(D_0)$ along this commutative algebra map to get a new rigid $C$-algebra, $D$, and observe that $C\to D$ is fully faithful and $D$ is $\kappa$-compact. 

We claim that the composite $$\Sigma^n\Sph[BS^1]\to  K(D_0^\at)[BS^1] \to K(D^\at)[BS^1] = K(D^\dbl)[BS^1]\xrightarrow{\Dim} \End(\one_C)$$ is $x$, which will prove the desired claim. Here, the first arrow is a lift to $K(D_0^\at)$ of the canonical map in $\pi_n\map_{\Mot_C}(M_C(C),M_C(D_0))$ which exists by assumption. 

This in fact follows from \Cref{cor:commutativediagram}: indeed, this states that the relevant composite is equivalent to $K(D_0^\at)[BS^1]\to \udl{\HH(D_0/C)_{hS^1}}\to \End(\one_{\Frrig_C(D_0)})\to \End(\one_C)$, but by virtue of factoring through $D$, this last morphism is exactly free on the map $x: \HH(D_0/C)_{hS^1}\simeq \Sigma^n\one_C[BS^1] \xrightarrow{x} \one_C$. 
\end{proof}

\begin{cor}[Motivic surjectivity]\label{lm:motsurj}
Let $C$ be as in \Cref{nota:std}, and let $x\in \pi_n\map_{\Mot_C}(M_C(C),M_C(C))$. There exists a $\kappa$-compact rigid extension $D$ of $C$ such that the image of $x$ in $\pi_n\map_{\Mot_D}(M_D(D),M_D(D))$ is in the image of $\pi_nK(D^\dbl)$.
\end{cor}
\begin{proof}
Let $M := \Sigma^n M_C(C)$. By combining \Cref{ex:atomicunit} and \Cref{ex:atomicshift}, we find that the canonical class in $\pi_n \map_{\Mot_C}(M_C(C),M)$ is $\kappa$-atomic. View $x$ as a map $M\to M_C(C)$. 

 Applying \Cref{ex:atomiccompos} now gives the desired $\kappa$-compact rigid extension $D$. 
\end{proof}

\begin{lm}[Atomic injectivity]\label{lm:atinj}
Let $C$ be as in \Cref{nota:std}, and let $x\in \pi_nK(C^\dbl)$ such that $\Dim(x) = 0\in\pi_n(\End(\one_C)^{BS^1})$. 

There exists a $\kappa$-compact rigid extension $D$ of $C$ such that the image of $x$ in $\pi_nK(D^\dbl)$ is $0$. 
\end{lm}
\begin{proof}
    Let $x\in \pi_n K(C^\dbl)$ be in the kernel of $\pi_n\Dim$, and let $\overline{x}$ be its image in $\End_{\Mot_C}(M_C(C))$.

Fix a $\kappa$-compact $C^\dbl$-module $E_0$ with $M_{C^\dbl}(E_0)\simeq \Sigma^n M_{C^\dbl}(C^\dbl)$. We let $E:= C\otimes_{\Ind(C^\dbl)}\Ind(E_0)$.

By \Cref{cor:almostrealizingclasses}, we can find a $\kappa$-compact rigid extension $D$ of $C$ such that after basechange to $D^\dbl$, $x$ can be realized as a $D^\dbl$-linear map $X:D^\dbl\otimes_{C^\dbl}E_0\to D^\dbl$. Up to changing notation, we set $C=D$ and assume $x$ can be realized as a map $E_0\to C^\dbl$. 

Note that the following composite is null essentially by design, since it is $\Dim\circ \overline{x}$: $$\map_{\Mot_C}(M_C(C),M_C(E))[BS^1]\to \map_{\Mot_C}(M_C(C),M_C(C))[BS^1]\xrightarrow{\Dim} \End(\one_C)$$

By \Cref{cor:commutativediagram}, we may rewrite it as $$\map_{\Mot_C}(M_C(C),M_C(E))[BS^1]\to \udl{\HH(E/C)_{hS^1}} \to \End(\one_{\Frrig_C(E)}) \to \End(\one_C)$$ where the second map is induced by the commutative $C$-algebra map $\tilde{X}: \Frrig_C(E)\to C$ induced by $X$. 

Now the map $\one_C\otimes \map_{\Mot_C}(M_C(C),M_C(E))[BS^1]\to \HH(E/C)_{hS^1} \simeq \Sigma^n \one_C[BS^1]$ is split (since the map $\one_C\otimes K_C(C)\to \one_C$ is) and hence we find that the map $\HH(E/C)_{hS^1}\to \one_C$ is actually null. 

Fix a $\kappa$-compact compactly generated $C^\dbl$-module $F_0$ with a fully faithful inclusion $E_0\to F_0$ and with trivial $C^\dbl$-motive, as is possible by \Cref{lm:disk}. Set $F:=C\otimes_{\Ind(C^\dbl)}\Ind(F_0)$.

Note that the $C$-linear motive of $F$ is also trivial, and in particular by \Cref{cor:unitinFrrig} the induced map $\End(\one_{\Frrig_C(E)})\to \End(\one_{\Frrig_C(F)})$ is equivalent to the map $\End(\one_{\Frrig_C(E)})\to \one_C$ from above. 

It follows that both $\Frrig_C(E)\to \Frrig_C(F)$ and $\Frrig_C(E)\to C$ factor through $$\Frrig_C(E)\otimes_{\End(\one_{\Frrig_C(E))})}\one_C.$$ The former map induces, by \Cref{prop:Frrigff}, a fully faithful map $$\Frrig_C(E)\otimes_{\End(\one_{\Frrig_C(E)})}\one_C\to \Frrig_C(F).$$
In total, the basechange $\Frrig_C(F)\otimes_{\Frrig_C(E)\otimes_{\End(\one_{\Frrig_C(E)})}\one_C}C$ is a $\kappa$-compact rigid extension of $C$, which we now call $D$.

Note that the map $E\to C\to D$ now factors through $F$, and so the map $E_0\to D^\dbl$ factors through $F_0$. 

On $K$-theory, this means that $K(E_0)\to K(D^\dbl)$ (whose image on $\pi_n$ hits the image of $x$ in $K(D^\dbl)$) factors through $K(F_0) = 0$. Hence $x=0$, as was to be shown. 

\end{proof}

And finally, we prove the following lemma (which, ultimately, we will not really use, but is probably good to have in handy): 
\begin{lm}[Motivic injectivity]\label{lm:motinj}
Let $C$ be as in \Cref{nota:std}, and $x\in \pi_n \map_{\Mot_C}(M(C),M(C))$, and suppose $\Dim(x)=0\in\pi_n(\End(\one_C)^{BS^1})$.  

There exists a $\kappa$-compact rigid extension $D$ of $C$ such that the image of $x$ in $\pi_n\map_{\Mot_D}(M(D),M(D))$ is $0$. 
\end{lm}
\begin{proof}
    This is essentially the same proof as above with \Cref{cor:almostrealizingmaps} instead of \Cref{cor:almostrealizingclasses}, and staying at the level of $C$ throughout . 
\end{proof}

We can now prove the main theorem of this section: 
\begin{proof}[Proof of \Cref{thm:designer}]
    Fix $C$ as in the statement, and let $\kappa$ be an uncountable regular cardinal such that $\one_C$ is $\kappa$-compact. 
    
       We apply the small object argument \cite[Proposition 1.4.7]{DAGX} to the class $S$ of rigid extensions in $\TwoRing^\rig_C$, and the subset $S_0$ of rigid extensions between $\kappa$-compact rigid commutative $C$-algebras. The class $S$ is indeed weakly saturated: fully faithful maps are closed under pushouts by \Cref{cor:fftensor}, they are obviously closed under retracts. For closure under transfinite composition, note that the functor taking compact objects $\TwoRing^\rig_C \to \Catperf$ commutes with filtered colimits and reflects fully faithfulness. In $\Catperf$, fully faithfulness is closed under filtered colimits and hence under transfinite compositions. 
    
    The result of \textit{loc. cit.} implies that we can factor $C\to 0$ as $C\to D \to 0$ where $C\to D$ is a rigid extension and $D\to 0$ has the right lifting property against $\kappa$-compact rigid extensions. 

By \cite[Lemma A.23]{ChroNS}, any $\kappa$-compact rigid extension $D\to E$ is obtained as the basechange of a $\kappa$-compact rigid extensions $D_0\to E_0$ between $\kappa$-compacts along a map $D_0\to D$, and therefore $D\to 0$ also has the right lifting property against all $\kappa$-compact rigid extensions of $D$. In other words, any $\kappa$-small extension of $D$ retracts onto $D$.

We can now prove that $D$ does the job: suppose $x\in \pi_n K(D^\dbl)$ vanishes under $\Dim$. By \Cref{lm:atinj}, there exists a $\kappa$-small rigid extension $D'$ of $D$ such that $x$ vanishes in $\pi_n K((D')^\dbl)$. Since $D\to D'$ has a retraction, $x$ must have already been $0$. For surjectivity on $\pi_n$, we use \Cref{lm:atomicsurj} instead. 

The same argument works for $K_C(D)$, using \Cref{lm:motinj} and \Cref{lm:motsurj} instead. 
\end{proof}

As a direct corollary, we obtain:
\begin{cor}\label{cor:levytext}
    For every commutative ring spectrum $R$, there exists a small rigid idempotent-complete commutative $\Perf(R)$-algebra $C$ such that $K(C)\simeq R^{BS^1}$, and $\End(\one_C)\simeq R$. 

    In particular, for every ordinary ring $R$, there exists a small rigid idempotent-complete commutative $\Perf(R)$-algebra $C$ such that $K_0(C)\cong R$. 
\end{cor}
This answers \cite[Conjecture 5]{levycat} by taking $R=\overline{\mathbb F_p}$, and more generally Question 16 in \textit{loc.cit.} (though, even in characteristic $0$, we lose control over the ``trace zero'' aspect of Levy's results - as explained in \cite[Theorem 4]{levycat}, in characteristic $p$ this is necessary). 

This leaves open the following question:
\begin{ques}
    Which commutative ring spectra can appear as the $K$-theory of a rigid stably symmetric monoidal category ?
\end{ques}

Employing the same method of versal objects as in \cite[Section 2]{levycat}, we also find: 
\begin{cor}
    There exists a functor $C_{-}: \mathrm{CRing}\to \CAlg^\rig_\mathbb Z$ and a natural isomorphism $K_0(C_{-})\cong \mathrm{id}$. 
\end{cor}
\section{Nullstellensatzian rigid categories}\label{section:NS}
Nullstellensatzian objects were introduced in \cite{ChroNS} as a way to axiomatize Hilbert's Nullstellensatz in categories that look like ``categories of rings'' but where it is not necessarily possible to reasonably speak of elements and polynomial equations. They feature here for several reasons. First, the strategy used in \cite{ChroNS} to prove redshift for commutative ring spectra is in some sense a winning strategy to understand redshift: either Nullstellensatzian objects satisfy redshift, in which case all objects do (redshift is ``downwards-closed''), or they do not, in which case they provide natural and canonical counterexamples. This idea is in fact how we originally arrived at the results of the present paper. As is perhaps now clear, we will see that $\aleph_1$-Nullstellensatzian rigid $T(n)$-local categories all fail redshift, thus proving truly canonical counterexamples. 

Second, the examples constructed in \Cref{section:main} are obtained via a small object argument which is almost the same as the one going into the construction of Nullstellensatzians, except that we restrict to $\kappa$-small extensions instead of all $\kappa$-small morphisms. Thus the objects we already considered are already some kind of Nullstellensatzians, though they don't exactly fit in the framework of \cite[Appendix A]{ChroNS}. It is therefore natural to wonder how actual Nullstellensatzians behave. 

Finally, Nullstellensatzian rigid categories are extremely unusual kinds of objects, that control some of the landscape of the theory of rigid categories, and of tt-geometry, as partly demonstrated in \cite{BHR}, and it therefore seems fruitful to obtain more information about them. The goal of this section is to gather consequences of the previous sections for Nullstellensatzians in general. In the next section, we will begin an investigation of Nullstellensatzian rigid categories at height $0$, that is, over $\mathbb Q$. 
\subsection{Nullstellensätze}\label{section:NSrec}
In this section, we review the notion of Nullstellensatzian objects from \cite{ChroNS}. The intuition behind the following definition is that one can phrase Hilbert's Nullstellensatz as the statement that over an algebraically closed field $K$, any nonzero, finite type $K$-algebra $R$ admits a map back to $K$. 

From this perspective, the following definition makes sense\footnote{It is at first sight perhaps too optimistic, but the results of \cite{ChroNS} do show that this is a reasonable definition.}: 
\begin{defn}
    Let $C$ be an accessible category. An object $x\in C$ is called $\aleph_0$-Nullstellensatzian if it is non-terminal, and for every non-terminal compact object of the slice, $y\in (C_{x/})^{\aleph_0}$, admits a map back to the initial object $x\in C_{x/}$. 

    More generally, $x$ is called $\kappa$-Nullstellensatzian if it is non-terminal and any non-terminal objects in $(C_{x/})^\kappa$ admits a map back to $x\in C_{x/}$. 
\end{defn}
\begin{ex}
    Hilbert's Nullstellensatz guarantees that algebraic closed fields are $\aleph_0$-Nullstellensatzian in the category of ordinary commutative rings. 

    More generally, Lang proves in \cite{Lang} that algebraically closed fields of cardinality $\geq \kappa$ are $\kappa$-Nullstellensatzian. 
\end{ex}
Since the word ``Nullstellensatzian'' is long to pronounce, read and write, we adopt the following convention:
\begin{conv}
We write ``$\kappa$-NS'' in place of ``$\kappa$-Nullstellensatzian''. 

Furthermore, we write ``NS'' for ``Nullstellensatzian'', which we use as a synonym for ``$\kappa$-NS for some $\kappa$'' rather than specifically $\aleph_0$-NS as is done in \cite{ChroNS} - this is because $\aleph_0$-NS objects will not play a specific role in this paper. 
\end{conv}
\begin{obs}
    Let $\kappa \leq \lambda$ be two regular cardinals. Since $\kappa$-compacts are $\lambda$-compact, we see that $\lambda$-NS objects are $\kappa$-NS as well. 
\end{obs}

NS objects truly make sense in contexts somewhat close to Hilbert's original Nullstellensatz, so in categories of objects that are somewhat ``like rings''. 

The following definition axiomatizes some desirable feature of such a category. 
\begin{defn}
    In a category $C$ with a terminal object $*$, the terminal object is called \emph{strict} if for any $x$, the existence of a map $*\to x$ implies that $x$ is also terminal (and in particular the map $*\to x$ is an equivalence).

    A presentable category $C$ is called \emph{weakly spectral} if it has a strict and compact terminal object. 
\end{defn}
\begin{rmk}\label[rmk]{rmk:nocg}
    The definition of weakly spectral categories in \cite{ChroNS} includes the requirement that they be compactly generated. However, we will only use this notion in order to apply \cite[Proposition A.17]{ChroNS}, the proof of which does not require compact generation, only presentability. Compact generation comes in at later points in \cite[Appendix A]{ChroNS}, for more subtle properties of their constructible spectrum.
\end{rmk}
\begin{ex}
    The category of ordinary commutative rings is weakly spectral \cite[Example A.6]{ChroNS}, and so is the category of commutative ring spectra \cite[Example A.9]{ChroNS} (see also \cite[Remark A.10]{ChroNS}). 
\end{ex}
\begin{ex}
    If $C$ is weakly spectral, so is $C_{R/}$ for any $R\in C$, \cite[Lemma A.15]{ChroNS}. 
\end{ex}
The key result here is \cite[Proposition A.17]{ChroNS}, guaranteeing a large supply of NS objects:
\begin{prop}[{\cite[Proposition A.17]{ChroNS}}]
Let $C$ be a weakly spectral category and $R\in C$ non-terminal. For any regular $\kappa$, there exists a $\kappa$-NS object in $C_{R/}$.     
\end{prop}
If we think of NS objects as ``algebraically closed'' objects, then this proposition says that every object maps to at least one ``algebraically closed'' object. Note that the assumption of weak spectrality is rather weak, and in particular one may really think of this result as encompassing all categories of ``ring-like objects''.

It is sometimes possible to \emph{classify} NS objects, and \cite{ChroNS} provides one striking example of such a classification in the context of $T(n)$-local commutative ring spectra. Their result can be stated as follows: 
\begin{thm}[{\cite[Theorem 6.12]{ChroNS}}]\label{thm:choNS}

\begin{enumerate}
    \item At height $0$, a rational commutative ring spectrum is $\kappa$-NS if and only if it is equivalent to $L[u^{\pm 1}], |u|=2$ for some algebraically closed field $L$ of characteristic $0$ of cardinality $\geq \kappa$; 
    \item At height $n\geq 1$ and an implicit prime $p$, a $T(n)$-local commutative ring spectrum is $\kappa$-NS if and only if it is equivalent to the Lubin--Tate theory $E(L, \mathbf G)$ for some algebraically closed field $L$ of characteristic $p$ and cardinality $\geq \kappa$, and $\mathbf G$ some\footnote{Up to isomorphism, there is at most one.} formal group of height $n$ over $L$.
\end{enumerate}
    
\end{thm}
We point out that the situation in positive ``non-chromatic'' characteristic, i.e. over $\mathbb F_p$, is more subtle, cf. \cite{FloNS} (see more concretely Proposition 1.4. in \textit{loc. cit.}). 

We will not (attempt to) reach as concrete a description in the context of rigid $2$-rings as the above, but our goal in the next section is to study some basic properties of NS rigid commutative algebras over some locally rigid base $C$.

\subsection{Basics about Nullstellensatzian rigid categories}\label{section:NSrig}
In this section, after quickly proving the existence of sufficiently many NS rigid algebras over a locally rigid base, we start analyzing their basic properties. 

We have mentioned in \Cref{section:NSrec} that most categories of ring-like objects are \emph{weakly spectral} in the sense of \cite[Definition A.1]{ChroNS}, and therefore admit a sufficient supply of NS objects. We briefly verify that $\TwoRing^\rig_C$ is indeed weakly spectral for most locally rigid $2$-rings $C$, subject to some mild finiteness condition:
\begin{lm}
Let $C$ be a locally rigid $2$-ring with a $\otimes$-conservative compact object $c$. 

The category $\TwoRing_C^\rig$ of rigid commutative $C$-algebras is weakly spectral.
\end{lm}
\begin{proof}
    We have to check that the terminal object $0\in\CAlg^\rig$ is strictly terminal and compact. While we do not need to (cf. \Cref{rmk:nocg}), we also prove that $\TwoRing_C^\rig$ is compactly generated. 

    For the latter claim, we simply observe that the inclusion $\TwoRing_C^\rig \subset \CAlg(\Mod_{C,\omega})$ preserves all colimits (hence filtered ones), and the forgetful functor $\CAlg(\Mod_{C,\omega})\to \Mod_{C,\omega}$ preserves filtered colimits. It follows that the left adjoint $\Frrig_C : \Mod_{C,\omega}\to \TwoRing^\rig_C$ preserves compacts, and since the forgetful functor is conservative, these compacts generate $\TwoRing^\rig_C$ under colimits, thus proving compact generation. 

    That $0$ is strictly terminal follows from the fact that it is strictly terminal in $\CAlg(C)$ and that if $\End(\one_D) = 0$, then $D=0$ for any pointedly monoidal $D$. 

    That $0$ is compact follows from the fact that $\Mod_{C,\omega}\xrightarrow{(-)^{\aleph_0}}\Catperf$ preserves filtered colimits, and that if $\one_D\otimes c=0$ for some rigid commutative $C$-algebra $D$ and our compact generator $c$, then $D=0$. 
\end{proof}
\begin{defn}\label{defn:quasicompact}
    A locally rigid $2$-ring is called quasi-compact if it has a $\otimes$-conservative compact object. 
\end{defn}
\begin{rmk}
    It may be that a better definition is simply that the $0$ object is compact in $\CAlg$, and that the above is only a convenient criterion to prove this. 
\end{rmk}
\begin{ex}
    The category of $T(n)$-local spectra is quasi-compact, since the $T(n)$-localization of any finite type $n$ spectrum is a $\otimes$-conservative compact object. 
\end{ex}
\begin{cor}\label{cor:existNS}
Let $C$ be a quasi-compact locally rigid $2$-ring.
    For any nonzero $D\in\TwoRing^\rig_C$ and any regular cardinal $\kappa$, there exists a map $D\to E$ where $E$ is $\kappa$-NS in $\TwoRing^\rig_C$. 
\end{cor}
\begin{proof}
    This follows from the previous lemma by \cite[Proposition A.17]{ChroNS}.
\end{proof}
\begin{rmk}
    The fact that compact generation is not needed, cf. \Cref{rmk:nocg}, shows that one may also consider \emph{large} rigid categories in the sense of Gaitsgory--Rozenblyum \cite[Chapter 1, §9]{GR}, and this category \emph{also} admits a sufficient supply of NS objects (here, we do use presentability, cf. \cite[Theorem A]{LocRig}, and that the terminal object is still compact follows from \cite[Proposition 2.52]{Dbl}). 
\end{rmk}
Thus, there is a large supply of $\kappa$-NS objects in $\TwoRing^\rig_C$ for ``most'' locally rigid $2$-rings $C$, and we are in a reasonable position to study them\footnote{If one is only interested in failure of redshift, our proof is actually in some sense fully constructive, and one does not need to consider specifically NS objects.}. We will in fact study those NS objects also without the quasi-compactness assumption, although when removing this condition, what we prove might be vacuous. 

The first basic observation is that they are related to NS commutative rings by decategorification as follows:

\begin{lm}\label{lm:end1NS}
Let $C$ be a locally rigid $2$-ring, and let $D$ be a $\kappa$-NS rigid commutative $C$-algebra, for any regular $\kappa$. The commutative algebra $\End(\one_D)\in \CAlg(C)$ is $\kappa$-NS. 
\end{lm}
\begin{proof}
    Let $\End(\one_D)\to R$ be a $\kappa$-compact, nonzero commutative $\End(\one_D)$-algebra in $C$. 

    We can form the pushout $D_R:=D\otimes_{\End(\one_D)}R$ (cf. \Cref{nota:basechange}). Since $\Mod_{\End(\one_D)}(C)\to D$ is fully faithful, so is $\Mod_R(C)\to D\otimes_{\End(\one_D)}R$ by \Cref{cor:fftensor}. Since $R\neq 0$, it follows that $D_R\neq 0$. Furthermore, $D_R$ is $\kappa$-compact over $D$ and thus there exists a splitting $D_R\to D$, which provides the appropriate splitting on units. 
\end{proof}

Let us recall that when $C=\Mod_\Q$, this means that $\End(\one_D)$ is a $2$-periodized algebraically closed field of characteristic $0$, $L[u^{\pm 1}], |u|=2$, $L$ of cardinality $\geq \kappa$, and when $C=\Sp_{T(n)}$, this means that $\End(\one_D)$ is a Lubin--Tate theory of height $n$ associated to an algebraically closed field of characteristic $p$, cf. \Cref{thm:choNS}.
\subsection{A $K$-theoretic Nullstellensatz}
In this section, we prove the following Nullstellensatzian analogue of \Cref{thm:designer}: 
\begin{thm}\label{thm:KNS}
     Let $C$ be a locally rigid $2$-ring with $\kappa$-compact unit, and let $D\in \TwoRing^\rig_C$ be $\max(\kappa, \aleph_1)$-NS. The Dimension morphism $$\Dim: K(D^\dbl)\to \End(\one_D)^{BS^1}$$ is an equivalence of commutative ring spectra. 
\end{thm}
As in \Cref{section:main} (cf. \Cref{cor:noshift}), we obtain the following immediate corollary: 
\begin{cor}\label{cor:NSnoshift}
    For $D\in\TwoRing^\rig_{T(n)}$ $\aleph_1$-NS, $L_{T(n+1)}K(D^\dbl)= 0$. 
\end{cor}
In fact this corollary follows also from \Cref{cor:smallnoshift}, and it does not strictly speaking provide more examples of failure of redshift, but it indicates that this failure happens ``everywhere'', and on canonical objects (Nullstellensatzians being the analogue of algebraically closed fields). 

Let us also recall some examples: 
\begin{ex}\label{ex:KNSTn}
    Suppose $C=\Sp_{T(n)}$ for some $n\geq 1$, and some implicit prime $p$. In this case, for any $\aleph_1$-NS rigid $C$-algebra $D$, $\End(\one_D)$ is a Lubin--Tate theory $E_n(k)$, where $k$ is an algebraically closed field of characteristic $p$ (and we have suppressed the formal group of height $n$ from the notation since over such a field, it is unique up to isomorphism). Thus we find that $K_0(D^\dbl)\cong E_n(k)^*(BS^1)\cong W(k)\llbracket u_1,...,u_{n-1}\rrbracket \llbracket x \rrbracket$. 

    If $C=\Sp_\mathbb Q\simeq D(\mathbb Q)$, $\End(\one_D)\simeq L[u^{\pm 1}]$ for some algebraically closed field $L$
 of characteristic $0$ and $u$ in degree $2$, and so $K_0(D^\dbl) \cong L\llbracket x\rrbracket$.
 \end{ex}
 \begin{rmk}
     The above result, even in the rational case, is for $\aleph_1$-NS rigid categories. In fact, if $C$ is rigid and we wish to study $\aleph_0$-NS rigid commutative $C$-algebras, then it is not difficult to change the arguments to prove that $\dim : K(D^\dbl)\to \End(\one_D)$ is $\pi_0$-surjective. 

    More generally, for any compact space $X$ and map $X\to BS^1$, the restricted $\Dim$ morphism $\Dim_{\mid X}: K(D^\dbl)\to \End(\one_D)^X$ is $\pi_0$-surjective on the image of $\pi_0(\End(\one_D)^{BS^1})\to \pi_0(\End(\one_D)^X)$. In particular, if $\one_D$ is complex oriented, we get surjectivity onto $\pi_0(\End(\one_D)^{\mathbb CP^n})$ for every $n$, and thus the image in $\pi_0(\End(\one_D)^{BS^1})$ is dense for a certain complete topology. 
 \end{rmk}
 A somewhat surprising corollary is the following size estimate: 
\begin{cor}
    Let $C$ be a $\kappa$-NS object in $\TwoRing^\rig_\Q$, $\kappa\geq \aleph_1$. In this case, $C$ has at least $\kappa^{\aleph_0}$ objects.
\end{cor}
\begin{proof}
    $K_0(C)$ is of cardinality $\leq \max(\aleph_0, |\pi_0(C^\simeq)|)$ and by \Cref{ex:KNSTn}, it has cardinality $\geq |L[[x]]| \geq \kappa^{\aleph_0}$. 
\end{proof}
This is surprising since $\kappa^{\aleph_0}$ may be $>\kappa$, in which case $C$ cannot be $\kappa^+$-compact. This is in stark contrast with the case of commutative ring spectra over $\Q$, where the smallest $\kappa$-NS is of the ``minimum reasonable size'', namely $\kappa$ (and hence $\kappa^+$-compact). 

Let us now prove \Cref{thm:KNS}. There are many ways to go about this, and in particular it is possible to \emph{deduce} it from \Cref{thm:designer}. Let us instead, for simplicity, explain why the proof is the same:
\begin{proof}[Proof of \Cref{thm:KNS}]
Without loss of generality, assume $\kappa \geq \aleph_1$. 

    The proof of \Cref{thm:designer} only used one property of $D$: that any $\kappa$-compact rigid extension $D\to E$ split. If $D$ is $\kappa$-NS, then more is true: any rigid map $D\to E$ with $E\neq 0$ is split. But if $D\to E$ is a rigid extension and $D\neq 0$, then $E\neq 0$, so this property holds for $\kappa$-NS objects, and so the same proof applies directly. 
\end{proof}

We conclude this section with a brief remark, that will not be used in the rest of the paper but seems worth mentioning. 
\begin{rmk}
   Most of the technical results in \Cref{section:main} take the form ``For any locally rigid $2$-ring $C$, and given some kind of $\kappa$-small data, there exists a $\kappa$-compact rigid extension $D$ of $C$ satisfying some property''. When $C$ is itself $\kappa$-NS, the extension $D$ actually splits, and for all these results, it actually implies that $C$ itself has that property. It is worth stating explicitly these results in that case to get a sense of how surprisingly behaved these objects are. To give an example,  \Cref{cor:almostrealizingmaps} specializes to: if $C$ is $\kappa$-NS, then any map $M_C(E)\to M_C(C)$ in $\Mot_C$, where $E$ is a $\kappa$-compact compactly generated $C$-module, can be realized by a map $E\to C$ of $C$-modules.

   This means we can think of NS rigid categories as ``motivically self-injective''.
\end{rmk}
\section{Rational rigid categories and semisimplicity}\label{section:Q}
In this section, we study in more detail rational Nullstellensatzian rigid categories. 

One way to think about \Cref{thm:KNS} is that it is established by proving that certain extensions of rigid categories, constructed in terms of Dimension-theoretic data, are nonzero; and this nonzero-ness is much weaker than what is actually proved, namely fully faithfulness. 

There is another way of guaranteeing conservativity of certain operations, that is, of guaranteeing that they produce nonzero objects, which is semi-simplicity. The remarkable thing is now that in characteristic $0$, many free rigid categories turn out to be semi-simple, by a result of Deligne \cite{deligne}, which was already exploited in \cite{BHR}. This allows us to study NS rigid categories over $\mathbb Q$ in much more depth, and this is the goal of this section. 

We begin with some brief recollections about Dimensions in characteristic $0$ and about semi-simplicity, which we use to describe the main examples of semi-simple categories we will be considering, and deduce a number of results about NS rigid rational categories. 
\begin{conv}
    Throughout this section, all $2$-rings will be rigid and rational. 
\end{conv}
Thus, equivalently, we will be considering small idemptent-complete symmetric monoidal stable categories in which all objects are dualizable. 
\subsection{Dimensions in characteristic $0$}\label{section:ratdim}
Recall that if $x\in C$ is a dualizable object in a symmetric monoidal category, its Dimension is a map $BS^1\to \End(\one_C)$. If $C$ is additive, this map extends canonically to $\Sph[BS^1]\to \End(\one_C)$, and if $C$ is rational, to $\Q[BS^1]$. Since $\End(\one_C)$ is a commutative ring spectrum, it in fact factors through $\Q\{BS^1\}$, and so we will benefit from an understanding of the latter: 
\begin{lm}\label{lm:freepi*}
    The free rational commutative ring spectrum $\Q\{BS^1\}$ on $BS^1$ has homotopy groups a polynomial ring $\Q[t_i,i\in \NN]$ with $|t_i|=2i$. 

    In particular it is free as a rational commutative ring spectrum on these classes $t_i$.
\end{lm}
\begin{proof}
    It is the free commutative ring spectrum over $\Q$ on the module $\Q[BS^1] \simeq \bigoplus_n \Q\cdot u_n$ where $u_n$ is dual to $u^n \in \pi_*(\Q^{BS^1})\cong \Q[u], |u|=-2$.

    If we call $t_i$ the image of $u_i$ under the map $\Q[BS^1]\to\Q\{BS^1\}$, the result follows. 
\end{proof}
The generator $t_0$ corresponds to the usual dimension, with no $S^1$-action, and will bear particular importance for us. Thus, we single it out with the following notation: 
\begin{nota}\label{nota:t}
    We let $t:= t_0\in\pi_0(\Q\{BS^1\})$. 
\end{nota}
In particular, the Dimension $\Dim(x)$ of a dualizable object $x$ in such a category is determined by a sequence $d_i \in \pi_{2i}(\End(\one_C)), i \in\NN$, and is given as a ring map on $\pi_*$ by sending a homogeneous polynomial $P\in\Q[t_i,i\in \NN]$ to $P((d_i,i\in\NN))$. 
\begin{nota}
    For $(C,x)$ as above, we let $\Dim(x)_i$ denote the $d_i$ from the above discussion, i.e. the image of $t_i$ under $\Q\{BS^1\}\to \End(\one_C)$.

    In particular, $\Dim(x)_0= \dim(x)$. 
\end{nota}

The following will be a convenient calculational tool:
\begin{lm}\label{lm:sumissum}
    Let $C$ be a rational additively symmetric monoidal category and $x,y\in C$ two dualizable objects. In this case, $\Dim(x\oplus y)_i= \Dim(x)_i+\Dim(y)_i$. 
\end{lm}
\begin{proof}
    Evaluation at the $t_i$'s is essentially the map $$\Map(BS^1,\End(\one_C))\simeq \Map_{\mathbb E_\infty}(\Q[t_i,i\in\NN],\End(\one_C))\to \Map_\Q(\bigoplus_n \Q\cdot u_n,\End(\one_C))$$
    But now this total composite is simply induced by the adjunction between anima and $\Q$-modules, and is therefore clearly an additive equivalence 
    $$\Map(BS^1,\End(\one_C))\simeq  \Map_\Q(\bigoplus_n \Q\cdot u_n,\End(\one_C))\simeq \prod_n \Omega^{2n}\End(\one_C)$$
\end{proof}
\subsection{Semi-simple categories}\label{section:sscat}
In this section, we do some brief recollections on semi-simple categories with the goal to prove a certain conservativity result for rings in semi-simple categories, and later modules over semi-simple categories. 
\begin{defn}
    An additive $1$-category $C$ is called semi-simple if for all $x\in C$, $\End_C(x)$ is a semi-simple ring. 

    An object in a semi-simple additive category is called simple if its endomorphism ring is a division algebra. 
\end{defn}
Recall that a semi-simple ring is isomorphic, via the Artin--Wedderburn theorem, to a finite product $\prod_i M_{n_i}(D_i)$ of matrix rings over division algebras. 
\begin{lm}\label{lm:simplesplit}
    Let $C$ be a semi-simple additive $1$-category and let $a$ be a simple object. 

    In this case, any nonzero morphism $a\to x$ has a retraction.  
\end{lm}
\begin{proof}
Let $f:a\to x$ be a nonzero morphism. Without loss of generality, $C$ is idempotent-complete and generated under direct sums and retracts by $a\oplus x$. Let $R = \End_C(a\oplus x)$ so that $C\simeq \mathrm{Proj}^{\mathrm{f.g.}}_R$. Since $R\cong \prod_i M_{n_i}(D_i)$ for some division rings $D_i$ and some integers $n_i$, there is an equivalence $C\simeq \prod_i \Vect_{D_i}$. The property at hand is clearly closed under products of the category $C$, so we may assume $C=\Vect_{D_i}$, where the assumption guarantees that $a=D_i$ and then the claim is standard linear algebra. 
\end{proof}

\begin{cor}\label{cor:semisnilp}
    Let $C$ be a small rigid $2$-ring such that $\ho(C)$ is semi-simple and $\one_C\in \ho(C)$ is simple. If $R,S\in \Alg(\Ind(C))$ are nonzero rings, then $R\otimes S$ is also nonzero. 
\end{cor}
\begin{proof}
Write $R\simeq \colim_i C_i$ as a filtered colimit of objects of $C_{\one_C /}= (\Ind(C)_{\one_C/})^{\aleph_0}$. If $R\otimes S=0$, then by compactness of $S$ as an $S$-module, the map $S\to S\otimes C_i$ is $0$ for some $i$. 

But the composite $\one\to C_i\to R$ is nonzero since $R$ is a nonzero ring, so $\one_C\to C_i$ is nonzero. By \Cref{lm:simplesplit}, it has a retraction, and hence $S=0$ since $(S\to S\otimes C_i) = S\otimes(\one_C\to C_i)$ has a retraction
\end{proof}
\begin{rmk}
    This is of course false if the unit is not simple, as witnessed e.g. by $C=\Perf(k)\times\Perf(k)$, $R=(k,0), S= (0,k)$. The above statement shows that this is essentially the only way a counterexample can arise. 
\end{rmk}

\begin{cor}\label{cor:semiscons}
    Let $C$ be a small rigid $2$-ring and assume $\ho(C)$ is semi-simple with simple unit. 

Let $D,E$ be two (small) $C$-modules. If $D,E\neq 0$, then $D\otimes_C E\neq 0$.
\end{cor}
\begin{proof}
    Let $d\in D, e\in E$ be nonzero objects. In this case, $\End_D(d),\End_E(e)$ are nonzero algebras in $\Ind(C)$. By \Cref{cor:semisnilp}, their tensor product $\End_D(d)\otimes\End_E(e)$ is also nonzero, and so by \Cref{lm:homsintensor} $\End_{D\otimes_C E}(d\otimes e)$ is nonzero as well. 
\end{proof}
We will need one more claim related to tensor products. First, recall:
\begin{defn}\label{defn:absss}
    A ring $R$ is called absolutely semi-simple if for any field extension $K\to L$ and any central map $K\to R$, $L\otimes_K R$ is semi-simple. 

    An additive $1$-category is called absolutely semi-simple if its endomorphisms rings are absolutely semi-simple. 
\end{defn}
\begin{prop}\label{prop:abssstensor}
    Let $C,D$ be stable, $K$-linear categories whose homotopy categories are absolutely semi-simple, where $K$ is a commutative ring spectrum with $\pi_*K$ a graded field. In this case, $\ho(C\otimes_K D)$ is semi-simple. 
\end{prop}
\begin{proof}
    Let $x\in C, y\in D$. Our assumption guarantees that $\pi_0(\End_C(x)\otimes_K\End_D(y))$ is still semi-simple. 

    Now let $x_0,x_1\in C, y_0,y_1\in C$, it is clear that $x_0\otimes y_0\oplus x_1\otimes y_1$ is a retract of $(x_0\oplus x_1)\otimes (y_0\oplus y_1)$ in $C\otimes_K D$. 

    Thus, the full subcategory of $\ho(C\otimes_K D)$ spanned by retracts of pure tensors is closed under finite direct sums, retracts, and shifts. Furthermore, it is semi-simple, so that any map $a\to b$ in this full subcategory has a cofiber of the form $c\oplus \Sigma d$ for some $c$ a retract of $b$ and $d$ a retract of $a$, and is therefore also in this full subcategory. 

    Therefore, this full subcategory, since it generates $C\otimes_K D$ under cofibers, is the whole thing, thus proving the claim. 
\end{proof}

We conclude with the key example of semi-simple small rigid $2$-rings we will consider, where the semi-simplicity is essentially due to Deligne \cite{deligne} and generalized to this context in \cite{BHR}: 
\begin{thm}\label{thm:semisimple}
    Let $f:\Q\{BS^1\}\to K$ be a map of commutative ring spectra such that $\pi_*K$ is a graded field and $f(t-n)\neq 0$ for all $n\in\Z$ (here, $t$ is as in \Cref{nota:t}). In this case, $\Frrig_\Q(\pt)_K := \Frrig_\Q(\pt)\otimes_{\Q\{BS^1\}}K$ has an absolutely semi-simple homotopy category with simple unit. 
\end{thm}
\begin{proof}
See \cite[§4.A]{BHR} for details. For convenience, we sketch a proof here. 

    It is clear from the cobordism hypothesis that for direct sums $S$ of the generators $X^{\otimes i}\otimes X^{\vee, \otimes j}$, $\End_{\Frrig_\Q(\pt)}(S)$ is free on a finite set over $\Q\{BS^1\}$, and hence the ring $\pi_0(\End_{\Frrig_\Q(\pt)_K}(S_K))$ only depends on $\pi_0(f)$. 

    Thus, to prove that these specific endomorphism rings are absolutely semi-simple, we may assume without loss of generality that $K$ is in degree $0$. In this case, it is clear that $\Frrig_\Q(\pt)_K$ is the stable envelope of the free rigid additive $K$-linear category on a dualizable object with dimension $f(t)$. For this, Deligne proves in \cite[Proposition 10.3, Théorème 10.5]{deligne} that this additive category is absolutely semi-simple. 

    It follows that for $S$ such a direct sum, or summand thereof, $\pi_0(\End_{\Frrig(\pt)_K}(S_K))$ is absolutely semi-simple. Since $K$ is a graded field and $\End_{\Frrig(\pt)}(S)$ is free on $\Q\{BS^1\}$ on a set, these endomorphism ring spectra are all either even periodic or in degree $0$. Either way, it follows from this that the collection of these $S$'s, their summands and suspensions, is closed under cofibers in $\Frrig(\pt)_K$, and hence is the whole category, thus proving the absolute semi-simplicity. 
\end{proof}
\begin{rmk}
    For $K$ being the ``fraction field'' of $\Q\{BS^1\}$, obtained by inverting all nonzero homogeneous elements in $\pi_*\Q\{BS^1\}$, our $\Frrig(\pt)_K$ is exactly $\mathbb{A}^1_\eta$ from \cite{BHR}. 
\end{rmk}

\subsection{Nonintegral objects}\label{section:nonint}
\Cref{thm:semisimple} suggests looking at a particular class of objects. 
\begin{defn}
A map $f:\Q\{BS^1\}\to K$ be a map of commutative ring spectra such that $\pi_*K$ is a graded field and $f(t-n)\neq 0$ is called an admissible map. 

A dualizable object $x\in C$, for $C$ a rational rigid $2$-ring is called $f$-transcendental if $C\otimes_{\Q\{BS^1\}} K$ is nonzero, where $\Q\{BS^1\}\to \End(\one_C)$ is given by $\Dim(x)$. 

It is called non-integral if it is $f$-transcendental for some admissible $f$. 
\end{defn}
\begin{rmk}
    An object $x$ being $f$-transcendental or non-integral only depends on $\Dim(x)$. 
\end{rmk}
\begin{rmk}\label{rmk:translation}
    With the above definition, \Cref{thm:semisimple} states that if $f:\Q\{BS^1\}\to K$ is admissible, then $\Frrig(\pt)_K$ has an absolutely semi-simple homotopy category with simple unit.
\end{rmk}
\begin{rmk}\label{rmk:admissiblemaps}
    If $f:\Q\{BS^1\}\to K$ is an admissible map and $K\to K'$ is an arbitrary map between graded fields, the composite $f':\Q\{BS^1\}\to K\to K'$ is also admissible, and we can think of $f$ as ``refining'' this composite. $f$-transcendentality and $f'$-transcendentality are then equivalent. 

    By intrinsic formality of graded fields in characteristic $0$, any admissible map is ``refined'' in this sense by some canonical map $\Q\{BS^1\}\to K(\mathfrak p)$ given on homotopy groups by $\pi_*(\Q\{BS^1\})\to \pi_*(\Q\{BS^1\})_\mathfrak p/\mathfrak p$ where $\mathfrak p$ is a homogeneous prime ideal in $\pi_*(\Q\{BS^1\})$, not containing any $(t-n), n \in\Z$. These $K(\mathfrak p)$ are all $\aleph_1$-compact over $\Q\{BS^1\}$. 

    This gives a classification of admissible maps up to refinement, and hence of notions of transcendentality. 
\end{rmk}
 
\begin{obs}\label{obs:fcan}
    Let $C$ be a $\aleph_1$-NS in $\TwoRing^\rig_\Q$, and let $x\in C$ be $f$-transcendental for some admissible $f:\Q\{BS^1\}\to K$. If $K$ is $\aleph_1$-compact, then $\Dim(x): \Q\{BS^1\}\to \End(\one_C)$ factors through $K$. 

    In particular,  if $f$ has kernel $\mathfrak p$ on homotopy groups, then $\Dim(x)$ factors through $K(\mathfrak p)$. Since the map $K\to \End(\one_C)$ is injective on homotopy groups ($K$ is a graded field!), the kernel $\mathfrak p$ agrees with the kernel of $\Dim(x)$. Thus, the ``minimal'' (in the sense of \Cref{rmk:admissiblemaps}) admissible $f$ for which $x$ is $f$-transcendental only depends on $\Dim(x)$. 
\end{obs}
Indeed, the non-zeroness of $C\otimes_{\Q\{BS^1\}}K$ guarantees a splitting, for $C$ $\aleph_1$-NS. 
\begin{rmk}\label{rmk:factorimpladm}
    If $\Dim(x)$ factors through an admissible $f:\Q\{BS^1\}\to K$, then clearly $x$ is $f$-admissible! 
\end{rmk}

We will now show that non-integral objects have many remarkable properties, especially in NS categories. 

We begin with divisibility properties, which were the key to an earlier proof of our Noshift theorem in characteristic $0$: 
\begin{lm}\label{lm:transpdiv}
Let $C$ be a small rational rigid $2$-ring and let $x\in C$ be non-integral. For any integer $n\geq 1$, the category $C[\frac{x}{n}]$ obtained by freely adding to $C$ a dualizable object $y$ with an equivalence $\oplus_n y \simeq x$ is nonzero. 
\end{lm}
\begin{proof}
    This category is the pushout 
    \[\begin{tikzcd}
	{\Frrig_\Q(\pt)} & C \\
	{\Frrig_\Q(\pt)} & {C[\frac{x}{n}]}
	\arrow[from=1-1, to=1-2]
	\arrow[from=1-1, to=2-1]
	\arrow[from=1-2, to=2-2]
	\arrow[from=2-1, to=2-2]
\end{tikzcd}\] where the left vertical functor classifies the object $\oplus_n x$ where $x\in \Frrig_\Q(\pt)$ is the universal dualizable object. 

We claim that if $x$ is $f$-transcendental, for $f:\Q\{BS^1\}\to K$ admissible, then the pushout $$C_K := \Frrig_\Q(\pt)_K\otimes_{\Frrig_\Q(\pt)} C$$ is nonzero. Indeed, it is equivalently the pushout $$C\otimes_{\Q\{BS^1\}} K $$  which is nonzero by definition. 

Furthermore, the other basechange, namely $\Frrig_\Q(\pt)\otimes_{\Frrig_\Q(\pt)} \Frrig_\Q(\pt)_K$ is also nonzero. This is also given by extension of scalars at the level of units, hence it suffices to show that the pushout $\Q\{BS^1\}\otimes_{\Q\{BS^1\}}K$ is nonzero, where the map $\Q\{BS^1\}\to \Q\{BS^1\}$ is induced by $x\mapsto \oplus_n x$. 

By inductively applying \Cref{lm:sumissum}, we see that this map sends $t_i$ to $nt_i$ for all $i$, which is clearly an isomorphism since we are in characteristic $0$. Thus this pushout is nonzero. 

It follows from these facts and from \Cref{cor:semiscons}, which applies by \Cref{thm:semisimple} (see also \cite{BHR}), that the basechange $C[\frac{x}{n}]\otimes_{\Q\{BS^1\}}K$ is nonzero, and so in particular $C[\frac{x}{n}]$ is nonzero. 
\end{proof}
The following is the remarkable consequence, showing how large nonintegral objects have to be in NS rational rigid $2$-rings:
\begin{cor}\label{cor:NStransdiv}
    Let $C$ be NS rigid and $x\in C$ be non-integral. There exists a $y$ and an equivalence $x\simeq \oplus_n y$.  
\end{cor}
\begin{proof}
    By the previous lemma, $C[\frac{x}{n}]$ is nonzero and it is clearly compact so that the NS property of $C$ gives a splitting $C[\frac{x}{n}]\to C$ and hence the desired $y$. 
\end{proof}
Note that this already gives that for non-integral objects, their $K$-theory class is $n$-divisible. 
\begin{rmk}
    One can wonder whether this holds for more general objects $x$, i.e. whether the map $\Frrig_\Q(\pt)\to \Frrig_\Q(\pt)$ picking out an $n$-fold direct sum of the generator is more generally ``conservative'' in some sense. This is not the case in full generality: the endomorphism ring of the unit of any symmetric monoidal $C$ is commutative, and hence cannot be a nonzero matrix ring. It follows that in no nonzero symmetric monoidal $C$ is $\one_C$ an $n$-fold direct sum of a single object $x$, $n\geq 2$. More generally, $\bigoplus_k \one_C$ and $\bigoplus_k \Sigma\one_C$ have as endomorphism rings $M_k(R)$, for $R$ a commutative ring, and hence cannot be more than $k$-divisible under direct sum. This shows that for conservativity, one needs the elements $(t-n), n\in\Z$ to be inverted, and this is what the definition of non-integral encodes. 
\end{rmk}

Together, the following two lemmas then give us an alternative proof of Noshift at height $0$:  
\begin{lm}\label{lm:existtrans}
    Let $C$ be an $\aleph_1$-NS rational rigid $2$-ring. There exists a non-integral object in $C$. 
\end{lm}
\begin{proof}
The rational free commutative ring spectrum $\Q\{BS^1\}$ is $\mathbb Q[t_i, i \in \mathbb N]$ where $|t_i| = 2$. Let $f:\Q\{BS^1\} \to K$ be the admissible morphism obtained by inverting all homogeneous nonzero elements in the source. In this case, $\Frrig_\mathbb Q(\pt)_K$ is $\aleph_1$-compact, and so the inclusion $C\to C\otimes\Frrig_\mathbb Q(\pt)_K$ admits a retraction (the target is evidently nonzero since the tensor product is happening over $\mathbb Q$). The image under such a retraction of the generator in $\Frrig_\mathbb Q(\pt)_K$ is then a non-integral object, by \Cref{rmk:factorimpladm}.

    
\end{proof}
\begin{rmk}
    In fact, this is true for $\aleph_0$-NS rational rigid $2$-rings as well. Indeed, one can instead simply consider $\Frrig_\Q(\pt)\otimes_{\Q\{BS^1\}} \Q\{BS^1\}/(t_0-r)$ for any non-integer rational number $r$ and tensor this with $C$. Since this is compact and nonzero, its tensor with $C$ is also compact over $C$ and nonzero, hence has a map to $C$. Then there exists some (homogeneous) residue field of $\Q[t_i, i\in\NN]/(t_0-r)$ for which the basechange will remain nonzero, proving that the image in $C$ of the free element in $\Frrig_\Q(\pt)$ is nonintegral. 
\end{rmk}
\begin{lm}\label{lm:trans+1istrans}
Let $C$ be a rational rigid $2$-ring, and let $x\in C$ be a dualizable object.
    If $x$ is non-integral, then so is $\one_C\oplus x$. 
\end{lm}
\begin{proof}
By \Cref{lm:sumissum}, $\Dim(\one_C\oplus x)$ sends $t_i$ to $\Dim(x)_i$ if $i> 0$ and to $1+\dim(x)$ for $i=0$. Indeed, $\Dim(\one_C): BS^1\to \End(\one_C)$ factors through $\pt$: the $S^1$-action is trivial. 

If $f:\Q\{BS^1\}\to K$ is admissible, then the map $g:\Q\{BS^1\}\to K$ determined by sending $t_0$ to $1+f(t_0)$ and $t_i$ to $f(t_i)$ for $i>0$ is also admissible. If $x$ is $f$-transcendental, then, since $g$ and $f$ differ by an automorphism of $\Q\{BS^1\}$ (the one fixing all $t_i$'s and sending $t_0$ to $1+t_0$), $\one_C\oplus x$ is $g$-transcendental. 
\end{proof}

We thus obtain a second proof at height $0$ of the failure of redshift:
\begin{cor}
    Let $C$ be an $\aleph_1$-NS rigid rational $2$-ring. $K_0(C)$ is rational and hence so is $K(C)$. In particular, $L_{K(1)}K(C) = 0$: there is no redshift. 
\end{cor}
\begin{proof}
    It suffices to prove that $1\in K_0(C)$ is divisible by $n$ for any $n\geq 1$. 

    Let $x\in C$ be non-integral, which exists by \Cref{lm:existtrans}. Then by \Cref{lm:trans+1istrans}, $\one_C\oplus x$ is also non-integral, so by \Cref{cor:NStransdiv}, they are both $n$-divisible in $K_0(C)$, and therefore so is $1= 1+[x]-[x]$. 
\end{proof}
As follows from \Cref{lm:end1NS} and \Cref{thm:choNS}, the unit of a ($\aleph_1$-)NS rigid rational $2$-ring is a $2$-periodized algebraically closed field of characteristic $0$. One may therefore wonder how far these objects, and therefore the failure of redshift, are from classical mathematical practice. We now show that this $2$-periodicity can be removed, and in fact, we can obtain similar results in a completely discrete setting: 
\begin{cor}\label{cor:1cat}
    There exists a rational additively symmetric monoidal $1$-category with all objects dualizable whose $K$-theory is rational. 
\end{cor}
By taking the stable envelope of such a $1$-category, and using the theorem of the weight-heart \cite[Theorem 8.1.3]{hebestreitsteimle}, we can also find a rigid rational $2$-ring with a discrete unit whose $K$-theory is rational. 
\begin{proof}
Let $C$ be an $\aleph_1$-NS rigid rational $2$-ring. 

    What we have actually proved in the previous theorem is a relation of the form $\one_C\oplus~y^{\oplus n}\simeq~z^{\oplus n}$ in $C^\dbl$.
    
    This is an entirely additive relation, and in particular it also holds in $\ho(C)$, and hence $K^{\oplus}(\ho(C))$ is rational and $\ho(C)$ is our example. 
\end{proof}

We already know by \Cref{thm:KNS} that $\Dim$ induces an isomorphism for $\aleph_1$-NS rigid rational $2$-rings, so if two non-integral objects have the same Dimension, they have the same $K$-theory class. In fact, much more is true: 
\begin{lm}\label{lm:uniquetrans}
    Let $C$ be an $\aleph_1$-NS rational rigid $2$-ring and let $x,y\in C$ be two objects. Suppose $\Dim(x)= \Dim(y)$ and both are $f$-transcendental for some admissible $f$ (which can be chosen to be the same for $x$ and $y$ by \Cref{obs:fcan}).  In this case, $x\simeq y$. 

    In other words, there is only one non-integral object per dimension function. 
\end{lm}
\begin{proof}
    Choose the witness $f:\Q\{BS^1\}\to K$ to be the same for $x$ and $y$, and $\kappa$-compact, as is possible thanks to \Cref{rmk:admissiblemaps} and \Cref{obs:fcan}. By \Cref{obs:fcan}, the maps $\Frrig_\Q(\pt)\to C$ classifying $x$ and $y$ factor through $\Frrig_\Q(\pt)_K$. 

    Furthermore, we note that $\Frrig(\pt)_K\otimes_{\Perf(K)}\Frrig(\pt)_K$ is still semi-simple. Indeed by \Cref{thm:semisimple}, $\Frrig(\pt)_K$ is absolutely semi-simple in the sense of \Cref{defn:absss}. Thus, \Cref{prop:abssstensor} guarantees the semi-simplicity of the tensor product. 

    Furthermore, the unit in the tensor product still has $K$ as endomorphism ring spectrum, and hence the tensor product still has a simple unit.

    Hence, the pushout $\Frrig(\pt)_K\otimes_{\Frrig(\pt)_K\otimes_{\Perf(K)}\Frrig(\pt)_K}C$ is nonzero by \Cref{cor:semiscons}, and it is clearly $\aleph_1$-compact over $C$. 

    The retraction one obtains from $C$ being $\aleph_1$-NS is a witness that $x\simeq y$. 
\end{proof}
Thus non-integral elements are easy to classify in an NS category. The same cannot be said for elements of integer dimension, because of the usual counterexamples: 
\begin{ex}\label[ex]{ex:trivdim}
    Let $C$ be a rigid category and $x\in C$. Then $\dim(x\oplus\Sigma x) = 0$.
\end{ex}
\begin{rmk}
Of course in the above example, it is not only the dimension of $x\oplus\Sigma x$ which is $0$/an integer, it is its $K$-theory class, and we know from \Cref{thm:KNS} that this is all that can happen. Thus, if $\Dim(x) = \Dim(y)$, then $[x]=[y]$. In fact, we can get better: if $x,y$ have the same Dimension and are not necessarily non-integral, by \Cref{lm:existtrans}, in the $\aleph_1$-NS case we can find some non-integral $N$ such that $N\oplus x, N\oplus y$ are both non-integral, and of the same dimension. By \Cref{lm:uniquetrans}, they are therefore equivalent. 
\end{rmk}
\begin{cor}\label{cor:summandtrans}
    Let $C$ be an $\aleph_1$-NS rigid rational $2$-ring, and let $N$ be a non-integral object, $x$ an arbitrary object. $x$ is a summand of $N$. 
\end{cor}
\begin{proof}
    $N\oplus x\oplus\Sigma x$ is non-integral since its Dimension is equivalent to that of $N$. Thus, \Cref{lm:uniquetrans} guarantees it is equivalent to $N$, and hence $x$ is a retract of $N$.  
\end{proof}

The results of this section suggest the following picture of an $\aleph_1$-NS rigid rational $2$-ring $C$: we have all the $\one_C^{\oplus n}, \Sigma \one_C^{\oplus n}, n\in\NN$, which are the ``standard'' objects and correspond to the ring part $\Perf(\End(\one_C))$ of $C$; then we have non-integral objects, these are classified up to equivalence by their Dimension. In particular, they are ``permuted'' through the standard objects: if $N_0, N_1$ are non-integral objects and $\Dim(N_0)+n= \Dim(N_1)$, then $N_0\oplus\one_C^{\oplus n}\simeq N_1$, and similarly for $-n$ and $\Sigma \one_C^{\oplus n}$. These objects are somewhat mysterious, but how they relate to one another is somewhat understandable, ultimately because of \Cref{cor:semiscons}. In particular, by combining \Cref{ex:trivdim}, \Cref{lm:sumissum} and \Cref{lm:uniquetrans}, we find that for any non-integral object $N$, every object $x\in C$ is a retract of $N$. 

The more mysterious objects are the objects with integer dimensions which are not ``standard''. They exist, e.g. $N\oplus\Sigma N$ for any  non-integral $N$. There are also some ``weirder'' ones, e.g. objects whose $\Dim(x)_0$ is an integer (e.g. $0$) but which have higher $\Dim(x)_i$'s. We cannot manipulate those as easily, though of course they are also retracts of non-integral objects. 
\subsection{Duality}\label{section:duality}
In this section, we study duality. As in the previous sections, the only objects about which we can really say anything are the non-integral objects. We first give a definition: 
\begin{defn}
    Let $C$ be a symmetric monoidal category and $x\in C$ a dualizable object. We say $x$ is self-dual if there is an equivalence $x\simeq x^\vee$, and we say $x$ is \emph{coherently} self-dual if the symmetric monoidal functor $\Cob\to C$ classifying $x$ factors through the unoriented ($1$d-)cobordism category $\Cob^{\mathrm{unor}}$.  
\end{defn}
\begin{rmk}
    Just as self-duality should really be \emph{the data} of the equivalence $x\simeq x^\vee$, coherent self-duality should really be the data of the factorization through $\Cob^{\mathrm{unor}}$. 
\end{rmk}
\begin{rmk}
    There is another possible definition of ``coherently self-dual'', involving the fact that for any symmetric monoidal category $C$ where every object is dualizable, $C$ upgrades canonically to a $C_2$-fixed point in $\Cat$ under the $\op$-action, cf. Harpaz's construction here \cite{Yonatan}. Since anima are also canonically fixed under this $C_2$-action, this provides a $C_2$-action on $C^\simeq$, and one can define $x$ to be coherently self-dual if it admits a lift to $(C^\simeq)^{hC_2}$. We will not need to know that these two perspectives are equivalent. 
\end{rmk}

\begin{prop}
    Let $C$ be an $\aleph_1$-NS rigid rational $2$-ring, and let $x\in C$ be an object. Consider the following statements: 
    \begin{enumerate}
        \item $x$ is self-dual in $C$; 
        \item For all odd $i$, $\Dim(x)_i \in \pi_{2i}\End(\one_C)$ is $0$
    \item $x$ is coherently self-dual in $C$;
        \item $\Dim(x)$ is $C_2$-equivariant as a map $BS^1\to \End(\one_C)$, where $BS^1$ has the orientation-reversing $C_2$-action.
        \end{enumerate} 
        In general, (1) implies (2), (3) implies (4), (3) implies (1) and (4) implies (2). When $x$ is non-integral, all implications can be reversed. 
\end{prop}
\begin{proof}
On the $1$d oriented cobordism category, orientation reversal is a symmetric monoidal functor sending the positively oriented point to the negatively oriented point. It follows from the universal property of $\Cob$ that for any symmetric monoidal category $D$ and any dualizable object $x\in D$, $\dim(x^\vee): BS^1\to\End(\one_D)$ is equivalent to $\dim(x)\circ \iota$, where $\iota:BS^1\to BS^1$ reverses orientation. 

We also note that on $\Q[BS^1]\simeq \bigoplus_i \Q\cdot u_i$, $\iota$ induces $u_i\mapsto (-1)^i u_i$. 

This proves that if (1) holds, i.e. $x\simeq x^\vee$, then $\dim(x)_i = (-1)^i \dim(x)_i$, so that for all odd $i$, $\dim(x)_i = 0$ (since $2$ is invertible), i.e. (2) holds. 

It's clear that the above also holds if we assume (4) in place of (1). 

(3) clearly implies (1), since the point in $\Cob^{\mathrm{unor}}$ is self-dual. 

Finally, in general, that (3) implies (4) is also clear because the map $\Cob\to \Cob^{\mathrm{unor}}$ induces the map $BS^1\to (BS^1)_{hC_2}$ on endomorpisms of the unit. 

We now assume that $t$ is non-integral, and suppose (2) holds. We shall prove (3), which suffices to conclude. 

In this case, let $f: \Q\{BS^1\}\to K$ be an admissible morphism such that $x$ is $f$-transcendental. (2) implies that $f(t_i) = 0$ for $i$ odd. It follows that $\Q[\Cob^{\mathrm{unor}}]\otimes_{\Perf((\FCom)_\Q(BS^1))} \Perf(K) \neq 0$, where $\Cob^\mathrm{unor}$ is the unoriented cobordism category. Indeed, the endomorphism ring of the unit in the latter is exactly $(\FCom)_\Q((BS^1)_{hC_2})$ whose homotopy ring is $\Q[t_i,i\textnormal{ even }]$. 

Since this category is $\aleph_1$-compact, it follows that the basechange $C\otimes_{\Frrig(\pt)_K}\Q[\Cob^{\mathrm{unor}}]_K$ maps back to $C$, proving (3). 
\end{proof}

Thus among non-integral objects, self-duality is completely understandable at the level of Dimension, in a computable way. 
\subsection{Invertibles}\label{section:picard}
In this section, we prove: 
\begin{prop}\label[prop]{prop:Pic}
    Let $C$ be an $\aleph_0$-NS rigid rational $2$-ring. The only nontrivial invertible object in $C$ is $\Sigma\one_C$,i.e. $\pi_0\Pic(C)\cong \mathbb Z/2\cdot [\Sigma \one_C]$.  
\end{prop}

For this, we start by analyzing the free rational small $2$-ring on an invertible object. 
\newcommand{\Kos}{\mathrm{Kos}}
\begin{lm}
    The free rational rigid $2$-ring on an invertible object is equivalent to the rationalized linearization of $\PPic(\Z)$, $\Perf(\Q)[\PPic(\Z)]$. 
    
    It splits as a product $A\times B$ where $A=\Perf(\Q)[\Z]$ is the category of finite graded $\Q$-modules, and $B=\Perf(\Q)[\Z]^{\Kos}$ is the same category but with the Koszul sign rule symmetric monoidal structure. In particular, $A$ and $B$ are both semi-simple with simple units. 
\end{lm}
\begin{proof}
    The free symmetric monoidal category on an invertible object is the symmetric monoidal groupoid $\Omega^\infty\Sph$. 

    In particular, this has a map to $\PPic(\Z)$ picking out the invertible object $\Sigma \Z$. One checks by direct inspection that this map is an isomorphism on $\pi_0,\pi_1$. Since the higher homotopy groups of $\Omega^\infty\Sph$ are finite, it follows that $\Perf(\Q)[\Omega^\infty\Sph]\to\Perf(\Q)[\PPic(\Z)]$ is an equivalence, from this, the first claim follows. 

For the second claim, we note that the endomorphism ring of the unit in $\Perf(\Q)[\PPic(\Z)]$ is $\Q[C_2]\cong \Q\times\Q$ with a map sending the generator $\sigma$ to $(1,-1)$. It follows that the category splits as $A\times B$. We now identify $A,B$ as desired. We do it for $A$, and the argument is entirely analogous for $B$: note that the element $\Q(1)$ ($\Q$ in grading $1$) is invertible in $\Perf(\Q)[\Z]$, and hence induces a map $\Perf(\Q)[\PPic(\Z)]\to \Perf(\Q)[\Z]$. The dimension of $\Q(1)$ is $1$, and hence on endomorphisms of the unit, the map $\Q[C_2]\to \Q$ sends the generator to $1$, and so the map $\Perf(\Q)[\PPic(\Z)]\to \Perf(\Q)[\Z]$ factors through the projection onto $A$ - it is then easy to verify that the induced map $A\to\Perf(\Q)[\Z]$ is an equivalence.  
\end{proof}
By basechange, the following is immediate:
\begin{cor}
    Let $C$ be a rational $2$-ring and $L$ in $C$ be invertible. $C$ splits as $C_1\times C_{-1}$, where the image of $L$ in $C_1$ has dimension $1$, and dimension $-1$ in $C_{-1}$. 

    Furthermore, the map $\Perf(\Q)[\PPic(\Z)]\to C_1$ classifying $L$ factors through $\Perf(\Q)[\Z]$, and to $C_{-1}$ it factors through $\Perf(\Q)[\Z]^\Kos$. 
\end{cor}
We can now prove the main result of this section:
\begin{proof}[Proof of \Cref{prop:Pic}]
    Let $C$ be as in the proposition, and let $L\in C$ be invertible. 

    The previous corollary provides a decomposition $C\simeq C_1\times C_{-1}$. One of $C_1, C_{-1}$ must be nonzero. Up to changing $L$ by $\Sigma L$, we may assume $C_1\neq 0$. 
    
    Since $C_1$ is clearly compact as commutative $C$-algebras, we find that the map $C\to C_1$ admits a splitting, which proves $C= C_1$. 

    Thus the map $\Perf(\Q)[\PPic(\Z)]\to C$ classifying $L$ factors through graded $\Q$-modules, which is a semi-simple category with simple unit. It further maps to $\Perf(\Q)$, sending the image of the free invertible object to $\Q$, so that by \Cref{cor:semiscons}, the tensor product $C\otimes_{\Perf(\Q)[\Z]}\Perf(\Q)$ is nonzero. Since it is compact as a commutative $C$-algebra, it retracts onto $C$ which proves $L\simeq \one_C$. 
\end{proof}
\bibliographystyle{alpha}
\bibliography{Biblio}
\end{document}